\documentclass{article}

\usepackage[margin=0.75in,includefoot,footskip=30pt,]{geometry}
\usepackage[T1]{fontenc}
\usepackage[utf8]{inputenc}
\usepackage[english]{babel}
\usepackage[natbibapa]{apacite}

\usepackage{amsmath}
\usepackage{amssymb}
\usepackage{amsfonts}
\usepackage{amsthm}
\usepackage{dsfont}
\usepackage{mathtools}

\usepackage{enumitem}
\usepackage[dvipsnames]{xcolor}
\usepackage[normalem]{ulem}

\newtheorem{theorem}{Theorem}
\newtheorem{lemma}[theorem]{Lemma}
\newtheorem{assumption}[theorem]{Assumption}
\newtheorem{definition}[theorem]{Definition}

\numberwithin{equation}{section}
\numberwithin{theorem}{section}

\newtheorem{remark}[theorem]{Remark}

\usepackage[onehalfspacing]{setspace}
\usepackage[breaklinks=true, hidelinks=true]{hyperref}
\usepackage{breakcites}

\DeclareMathOperator{\red}{red}

\newcommand{\A}{\mathcal A}
\newcommand{\C}{\boldsymbol{C}}
\newcommand{\D}{\mathcal D}

\renewcommand{\d}{\boldsymbol r}
\newcommand{\E}{\mathbb E}

\renewcommand{\H}{\mathcal H}
\newcommand{\imax}{{i_\text{max}}}

\newcommand{\Iw}{\mathcal{I}}

\newcommand{\N}{\mathbb N}
\renewcommand{\P}{\mathbb P}

\newcommand{\R}{\mathbb R}
\renewcommand{\S}{\mathcal{S}}

\newcommand{\Uw}{\mathcal{U}}
\newcommand{\Fw}{\mathcal{F}}
\newcommand{\Vw}{\mathcal V}

\newcommand{\X}{\mathcal X}
\newcommand{\g}{\wt g}
\newcommand{\sse}{\subseteq}
\newcommand{\middlemid}{\; \middle\vert\; }

\newcommand{\bb}{\boldsymbol{\beta}}
\newcommand{\z}{\boldsymbol{z}}
\newcommand{\PiH}{\boldsymbol{\Pi}_\mathcal{H}}
\newcommand{\PiCh}{\boldsymbol{\Pi}_\mathcal{H}}

\newcommand{\PiX}{\boldsymbol{\Pi}_\mathcal{X}}

\newcommand{\oPiH}{ \boldsymbol{\Pi}_{\widehat{\mathcal{H}}} }

\newcommand{\wPiH}{\boldsymbol{\Pi}_{\widetilde{\mathcal{H}}}}

\newcommand{\wPiX}{\boldsymbol{\Pi}_{\widetilde{\mathcal{X}}}}
\newcommand{\eps}{\varepsilon}

\let\wt\widetilde
\let\qed\relax

\title{Risk-Sensitive Partially Observable Markov Decision Processes as Fully Observable Multivariate Utility Optimization problems}

\author{Arsham Afsardeir \thanks{Fakult\"at Elektrotechnik und Informatik, Technische Universit\"at Berlin,  Marchstr.~23, 10587, Berlin, Germany. Partially supported by DFG project OB 102/29-1. }  \and 
Andreas Kapetanis\thanks{Institute of Mathematics, Technische Universit\"at Berlin, Str. des 17. Juni 136, 10623 Berlin, Germany.}
\and
Vaios Laschos  \thanks{ WIAS Berlin. For the biggest part, V.L. was supported by DFG project OB 102/27-1. For the completion of the work, V.L. was supported by DFG  under Germany´s Excellence Strategy – The Berlin Mathematics
Research Center MATH+ (EXC-2046/1, project ID: 390685689).} \and
Klaus Obermayer \thanks{Fakult\"at Elektrotechnik und Informatik, Technische Universit\"at Berlin,  Marchstr.~23, 10587, Berlin, Germany. Partially supported by DFG project OB 102/27-1. }}

\date{\today}

\begin{document}

\maketitle

\begin{abstract}
    We provide a new algorithm for solving Risk Sensitive Partially Observable Markov Decisions Processes, when the risk is modeled by a utility function, and both the state space and the space of observations is finite. This algorithm is based on an observation that the change of measure and the subsequent introduction of the information space that is used for exponential utility functions, can be actually extended for sums of exponentials if one introduces an extra vector parameter that tracks the "expected accumulated cost" that corresponds to each exponential. Since every increasing function can be approximated by sums of exponentials in finite intervals,  the method can be essentially applied for any utility function, with its complexity depending on the number of exponentials.
\end{abstract}


%under a specific but extensive class of risk-sensitive optimization criterion which made by the linear combination of different exponential Utility functions regarding different running costs
\section{Introduction}

In the classical theory of Markov Decision Processes (MDPs), one deals with controlled Markovian stochastic processes $(S_n)$ taking values on a Borel space $\mathcal{S}$. These processes are controlled via a series of actions $(A_n)$, according to a policy $\pi$, that changes the underlying state transition probabilities $\mathrm{P}(S_{n+1}|S_n,A_n)$ of $(S_n)$. The goal is to find a policy $\pi$ that optimizes the expected value $$\Iw_{N}(s_0,\pi) = \mathbb{E}_{s_{0}}^{\pi}\left[\sum_{n=0}^{N-1}\beta^{n}C(S_{n},A_{n},S_{n+1})\right],$$
where $\beta\in(0,1]$ is called \textit{discount factor}, {$\C: \S \times \A \times \S \to \R$ is the cost function, $s_0 \in S_n$ is the initial state } and $N\in\mathbb{N}\cup\{\infty\}$  (the case where $\beta=1$ and $N=\infty$ at the same time, will be excluded in this work).  The inclusion of risk-sensitivity and partial observability are natural extensions to this standard model. \medskip

In classical MDPs, one makes the assumption that the controlled process $(S_n)$ takes values on a set of states $\mathcal{S}$ which is always accessible to the controller. However, in several real-life applications the real state is not directly observable and only secondary information dependent on the state, can be observed. Partially Observable Markov Decision Processes (POMDPs) are a generalization of MDPs towards incomplete information about the current state. POMDPs extend the notion of MDPs by a set of observations $\mathcal{Y}$ and a set of conditional observation probabilities $Q(\cdot|s)$ given the ``hidden'' state $s \in \S$. {$Q(y|s)$ namely represents the probability of observing $y$ while being in state $s$.} In \textbf{risk-neutral POMDPs}, one can introduce a new state space, called \textbf{belief (state) space} $\mathcal{X}=\mathcal{P}(\mathcal{S})$, the set of probability measures on $\S$, and a stochastic process  $(X_{n})$ taking values in $\mathcal{X}$, such that $X_{n}(s)$  is the probability of $S_{n}$ being equal to the  ``hidden'' state $s \in \mathcal{S}$ at time $n$, conditioned on the accumulated observations and actions up to time $n.$ One can treat this new process on the belief space like a Completely Observable Markov Decision Process (COMDP) on $\X$ with classical tools, retrieve optimal or $\eps$-optimal polices (i.e. policies with expected value at most $\eps$-far from the optimal value), and then apply them to the original problem. It is remarkable that, \textbf{due to the linearity} of the expectation operator, the belief state is a so-called \emph{sufficient statistic}. Broadly speaking, a sufficient statistic carries adequate information for the controller to make an optimal choice at a specific point in time. It also allows to separate the present cost from the future cost through a Bellman-style equation. For an introduction to sufficient statistics, we refer to \citep{hinderer1970}.\medskip

 To introduce risk-sensitivity we will work with the classical theory of expected utility \citep{jaquette1973,  bauerle2014}, where one tries to optimize
\begin{equation}\label{firsteq}\Iw_{N}(s_0,\pi) = \mathbb{E}_{s_0}^{\pi}\left[U\left[\sum_{n=0}^{N-1}\beta^{n}C(S_{n},A_{n},S_{n+1})\right]\right],\end{equation}
for some increasing and continuous function $U:\mathbb{R}\to\mathbb{R}$. Note that the exponential utility function $U = \exp$ generates a performance index that belongs to several models of risk at the same time, and it has been extensively studied in many different settings \citep{borkar2002, cavazos2010, chung1987, dimasi2007, dupuis2019exit, fleming1997, hernandez1996a, howard1972, levitt2001}. As mentioned above, \citep{bauerle2014} treated problems with optimality criteria of form \eqref{firsteq}. \medskip

To our best knowledge, there is only partial progress \citep{marecki2010, bauerle2017partially,bauerle2017a, fan2018, fernandez1997,baras1997,hernandez1999,cavazos2005} when it comes to combining risk-sensitivity and partial observability, and most articles study the specific case of the exponential utility function. When risk is involved extra or alternative information is necessary to make an optimal decision. In the case of expected utilities additional information  on the accumulated cost is necessary to make an optimal choice, even if the true state is known to the controller \citep{bauerle2014, marecki2010}. A common workaround to this problem is to assume that the controller is aware of the running cost through some mechanism. For example, the controller observes either the running cost directly \citep{marecki2010}, or a part of the whole process that is responsible for the cost \citep{bauerle2017partially}. \citet{fan2018} study cost functions that depend on both observable quantities and beliefs. In \citet{bauerle2017a}, a general approach for treating problems with risk measured by a  utility function is introduced. It is shown there, that one  can use  probability measures on the product space $X\times\R$ as state space. This way, the authors end up with an MDP on the state space $\mathcal{P}(X\times\R).$ As we will explain in the sequel, we take a different route that is computationally less demanding in several cases.

\subsection{Our Contribution}

It was shown in \citet{bauerle2017a}, that by using an exponential function as the utility function, the problem space would shrink dramatically . This proposition is incline to the concept of \emph{information vector} which Cavazos-Cadena and Hernandez-Hernandez discussed \citep{cavazos2005}. Both papers show, in two different conceptual systems, that this property of exponential utility function which transforms summation of costs to the product of their utilities, can be exploited to provide sufficient statistics for decision making in a much smaller representation. Consequently the model with exponential utility function has a computational advantage, however, it losses its generality due to the much narrower range of utility functions it can accept. \medskip

To extend the ideas related to exponential utility models, in this work we employ the idea of multi-variate optimization and show that by applying exponentials on a finite set of different running costs, we can use the information space approach in a more general way. More specifically, our multi-variate exponential utility model is able to reproduce more complex utility functions and at the same time benefit from simplicity of exponential utilities according to computation burden as well. In our method the exponential running costs are independent from each other, therefore, each term can represent an independent component of a target multivariate utility function. These components can be seen either as building blocks of a target function's formulation (in the case of utility functions equal to sum of exponential terms) or more generally as the elements of function approximator series (in case of approximating the target function).\medskip

In comparison to the purely exponential case, the utility functions we treat can be more behaviorally plausible as well. In behavioral economics and finance disciplines it is a common approach to assume people to use a mapping from objective values to subjective utilities and subsequently either apply maximization on the \textit{expected value} of the utilities \citep{von1947theory} or their other distributional properties \citep{tversky1992advances} \citep{al2008note}. Consequently, the shape of people’s utility function has a significant effect on their attitude toward risky choices. The shape of human utility function \citep{mosteller1951experimental, kalyanaram1995empirical} and the effect of contextual parameters on that, like amount of wealth \citep{markowitz1952utility} and emotions \citep{bertram2021subjective} has been investigated by different experimental paradigms. For a review see \citet{edwards1954theory}. In their influential work, \emph{ prospect theory},  Kahneman and Tversky proposed an S-shaped utility function which is risk-averse (concave) in gains and risk-seeker(convex) in losses \citep{kahneman1979interpretation, kahneman2013prospect}. They also addressed a set of experiments that confirm different risk tendencies in gain and loss situations. As one can expect, the flexibility of an exponential utility model is not sufficient to produce different risk tendencies between loosing and winning situations by a single function. In contrast, privileging the computational advantage of exponential utility functions, our risk-sensitive model can also address this phenomenon by exploiting either utility functions which are defined by linear combinations of multiple exponential forms (like \emph{sinh(.)}) or utility functions which can be approximated by linear combinations of exponential terms in a specific interval of values (like Sigmoid function, see section 4). Therefore, it becomes possible to shape functions which have both positive and negative second order derivatives on the distinct intervals of their domain simultaneously. Capturing the dissimilarity of risk tendencies among losses and gains is a major advantage of the multi-variate model which gives us more explanatory ability in respect to behavioral modeling in comparison to exponential utility function. \medskip

In what follows we argue that it is actually possible to apply similar arguments to treat utility functions that are sums of exponentials, i.e. utility functions of the form
\begin{equation}\label{sumexp}
\widehat{U}(t)=\sum_{i=1}^{i_{\max}}w^{i}e^{\lambda^{i}t}.\end{equation}

With slight abuse of notation, we observe that for a probability distribution $\theta_{0}$ on $\mathcal{S},$ we can write
\begin{equation*}
\widehat{\Iw}_{N}(\theta_{0},\widehat{\pi})=\widehat{\mathbb{E}}_{\theta_{0}}^{\widehat{\pi}}\left[\sum_{i=1}^{i_{\max}}w^{i}e^{\lambda^{i}\left(\sum_{n=0}^{N-1}\widehat{C}(S_{n},A_{n})\right)}\right]=\sum_{i=1}^{i_{\max}}\widehat{\mathbb{E}}_{\theta_{0}}^{\widehat{\pi}}\left[w^{i}e^{\lambda^{i}\left(\sum_{n=0}^{N-1}\widehat{C}(S_{n},A_{n})\right)}\right].
\end{equation*}  
Then, similar to \citet{cavazos2005}, by applying a change of measure argument, we identify a new state space $\mathcal{X}=\mathcal{P}(\mathcal{S})^{i_{\max}}\times \mathcal{Y},$ a controlled transition matrix $P(x'|x,a),$ and a collection of running costs $C^{i}:\mathcal{X}\times\mathcal{A}\times\mathcal{X}\to\mathbb{R},$ that depend on the next stage as well, such that for the resulting completely observable controlled processes $(X_{n})$. With $x_{0}=(\theta_{0},\dots,\theta_{0},y_{0})$ and a fixed but arbitrary $y_{0}\in \mathcal{Y}$, we have
\begin{equation}\label{aha}
\widehat{\Iw}_{N}(\theta_{0},\widehat{\pi})=\Iw_{N}(x_{0},\pi)=\mathbb{E}_{x_{0}}^{\pi}\left[\sum_{i=1}^{i_{\max}}w^{i}e^{\lambda^{i}\left(\sum_{n=0}^{N-1}C^{i}(X_{n},A_{n},X_{n+1})\right)}\right].
\end{equation}
Note that $(\pi, \mathbb{E})$ and $(\widehat{\pi},\widehat{\mathbb{E}})$ are connected in a straightforward manner, see Section \ref{sct:rspomdp}.\medskip

To establish our method we exploit the results of \citet{bauerle2014} and extend that model to introduce multivariate utility functions $\mathcal{U}:\mathbb{R}^{d}\rightarrow\mathbb{R},$ which are component-wise monotone and each variable corresponds to a different running cost. This gives rise to the following optimality equation:
%giving rise to the following optimality metric, {extending its application from \eqref{firsteq} }
\begin{equation}\label{multicrit}
\Iw_{N}(x,\pi):=\E_{x}^{\pi}\left[\Uw\left(\sum_{n=0}^{N-1}\boldsymbol{\beta}^{n}\cdot\C(X_{n},A_{n},X_{n+1})\right)\right].
\end{equation}
Note that here $C$ is a vector of different running costs and the dot product denotes the point-wise multiplication. Criteria like \eqref{multicrit} can arise when one is trying to solve a multi-objective task, with different running costs, each of the costs contributing in a different manner to a total cost. As an example, one can think of a policy maker allocating tax money to different public sectors (education, infrastructure, health, etc). For each of them, we get a different cost which can be the position in the global chart or any other comparison metric. However, the total utility, depends on how much each government prioritizes each of these aspects, something that is encoded in the choice of the utility function.  In a similar manner to \citet{bauerle2014}, we augment the space to keep track of each accumulated cost term. Furthermore we prove that the finite time discounted problem converges to the infinite time one, without the extra assumption demanding that the utility function is either convex or concave, appearing in \citet{bauerle2014}.\medskip
\begin{remark}
	For any two utility functions $\mathcal{U}_{1},\mathcal{U}_{2}$ that are $\eps$-close on some interval 
	$[N\min_{s,u}\widehat{C}(s,u), N\max_{s,u}\widehat{C}(s,u)],$  an $\eps$-optimal policy for  $\mathcal{U}_{1}$ is a $2\eps$-optimal policy for  $\mathcal{U}_{2}$ . Therefore, one can apply the method to solve RSPOMDPs with utility functions that can be \textit{approximated} by functions of the form \eqref{sumexp}. One can easily show that this includes all increasing real functions, by approximating $F(t)=U(log(t)),$ by a polynomial $P(t)=\sum_{i=1}^{k}w_{k}x^{k}$ in the interval $[\exp(N\min_{s,u}\widehat{C}(s,u)),\exp( N\max_{s,u}\widehat{C}(s,u))],$ and the composing with the exponential on both sides.
\end{remark}

\bigskip
The rest of the paper is structured as follows:
 In section 2 our multi-variate utility function and its mathematical modeling and construction are presented.  Next, in section 3 we show and prove solution methods for this model and the utility functions from section 2 in both finite and infinite horizon cases. In section 4, we provide an extended version of a famous POMDP, the tiger problem, as a numerical example to explain the model and compare it with the general model of Bäuerle and Rieder. And finally, we discuss about the computational advantage of our model.\\

\section{RSPOMDPs for sums of exponentials}\label{sct:rspomdp}

In this section, we consider risk sensitve POMDPs with a class of utility functions that can be written as weighted sums of exponentials.  We will show that it is possible to reformulate the problem in terms of a multi objective risk sensitive MDP with a new performance index that, in turn, can be treated with tools described in the next section. 

\subsection{The original setting}\label{original setting}

We start by describing the model for a risk sensitive POMDP, i.e. $(\mathcal{S},\mathcal{Y},\mathcal{A},\widehat{P},Q,\widehat{U})$.  $\mathcal{S},\mathcal{Y},\mathcal{A}$ will be three \textbf{finite} sets equipped with the discrete topology. In the sequel, $\mathcal{S}$ is the \textit{hidden} state space, $\mathcal{Y}$ the set of observations, and $\mathcal{A}$ the set of controls. For every $a\in \mathcal{A},$ we define a transition probability matrix $\widehat{P}(a)=\left[\widehat{P}(s'|s;a)\right]_{s,s'\in S}.$ Finally, we denote by $Q=[Q(y|s)]_{y\in \mathcal{Y}, s\in\mathcal{S}}$ the signal matrix and by $\widehat{C}:\mathcal{S}\times \mathcal{A}\rightarrow \mathbb{R}$ the cost function.\medskip

Now, for each $n\in\mathbb{N},$ let $\widehat{\mathcal{H}}_{n}$ be the set of histories up to time $n,$ where $\widehat{\mathcal{H}}_{0}=\mathcal{P}(S),$ and $\widehat{\mathcal{H}}_{n}= \widehat{\mathcal{H}}_{n-1}\times \mathcal{A}\times \mathcal{Y}.$ We denote by $\oPiH := \left\{ \widehat\pi =\left(\widehat{f}_0,\widehat{f}_{1},\dots\right) \middlemid \widehat{f}_n : \widehat{\H}_n \to A, n\in\mathbb{N} \right\}$ the set of deterministic polices that are functions of the history $\widehat{h}_{n}=(\theta,a_{0},y_{1},\dots,a_{n-1},y_{n})$ up to time $n.$ Given $\theta\in\mathcal{P}(S),$
and $\widehat{\pi}\in\oPiH,$ due to the Ionescu-Tulcea theorem, there exists a unique measure $\widehat{\mathbb{P}}^{\widehat{\pi}}_{\theta}$ on the Borel sets of
$\Omega:=\mathcal{S}\times(\mathcal{A}\times \mathcal{S}\times \mathcal{Y})^{\infty}$ that satisfies:
\begin{equation*}
	\widehat{\mathbb{P}}^{\widehat{\pi}}_{\theta}(s_{0},a_{0},s_{1},y_{1},a_{1},\dots,a_{n-1},s_{n},y_{n}):=\theta(s_{0})
	\prod_{k=0}^{n-1}\left(\widehat{P}\left(s_{k+1}\vert s_{k};\widehat{f}_{k}\left(\widehat{h}_{k}\right)\right)Q\left(y_{k+1}\vert s_{k+1}\right)\right),
\end{equation*}
 The corresponding expectation operator is denoted by $\widehat{\mathbb{E}}_{\theta}^{\widehat{\pi}}.$ Finally, for each $n\in\mathbb{N},$ we define the $\sigma$-fields $\widehat{\mathcal{F}}_{n}$, $\widehat{\mathcal G}_{n},$ by
\begin{equation*}
	\begin{split}
		\widehat{\mathcal{F}}_{n}:=\sigma\left(\left(A_{k},Y_{k+1}\right),\ k=0,1,\dots,n-1\right),\hspace{14pt}
		\widehat{\mathcal{G}}_{n}:=\sigma\left(S_{0},\left(A_{k},S_{k+1},Y_{k+1}\right),\ k=0,1,\dots,n-1\right).
	\end{split}
\end{equation*}
It is straightforward to see that the set of policies $\oPiH,$ contains exactly the elements $\big(\widehat{f}_{n}\big)_{n\in\mathbb{N}},$ where $\widehat{f}_{n}$ are $\widehat{\mathcal{F}}_{n}$-measurable functions from $\widehat{\mathcal{H}}_{n}$ to $\mathcal{A}.$

\subsection{Utility functions that are sums of exponentials}
Let $\{\lambda^{i}, i = 1,\dots,\imax\}\sse\mathbb{R}\setminus\{0\}$ be a finite collection of risk parameters, and $\{w^{i}, i=1,\dots,\imax\}\sse\mathbb{R}$ be a collection of weights. We define the utility function $\widehat{U}:\mathbb{R}\to\mathbb{R}$ by 
\begin{equation}\label{expsum}
	\widehat{U}(t):=\sum_{i=1}^{i_{\max}}w^{i}e^{\lambda^{i}t},
\end{equation}
and introduce the performance index
\begin{equation}\label{fistperformanceindex}
	\widehat{\Iw}_{N}(\theta_{0},\widehat{\pi})=\sum_{i=1}^{i_{\max}}w^{i}\widehat{\mathbb{E}}_{\theta_{0}}^{\widehat{\pi}}\left[e^{\lambda^{i}\left[\sum_{n=0}^{N-1}\widehat{C}(S_{n},A_{n})\right]}\right],
\end{equation}
and the corresponding value fuction
\begin{equation}\tag{$\widehat{P}$}
\widehat{V}_{N}(\theta_{0}) :=	\inf_{\widehat{\pi}\in\oPiH}\widehat{\Iw}_{N}(\theta_{0},\widehat{\pi}).
\end{equation}
The goal is to minimize $\widehat{\Iw}_{N}(\theta_{0},\widehat{\pi})$ over all policies  $\widehat{\pi}\in\oPiH.$ 
We want to show that we can instead work on the completely observable risk sensitive MDP on the space $\mathcal{X}$ with performance index
\begin{equation}\label{eqn:ObjectiveSumsOfExp}
	\Iw_{N}(x_{0},\pi)=\sum_{i=1}^{i_{\max}}w^{i}\mathbb{E}_{x_{0}}^{\pi}\left[e^{\lambda^{i}\left[\sum_{n=0}^{N-1}C^{i}(X_{n},A_{n},X_{n+1})\right]}\right],
\end{equation}
for some reconstructed cost functions $C^i$ ,measure and expectation operator $\mathbb{P}^{{\pi}}_{\theta}$, $\mathbb{E}^{{\pi}}_{\theta}$
and corresponding value function
\begin{equation}\tag{$P$}
\Vw_{N}(x) := \inf_{\pi\in\PiH}\Iw_{N}(x,\pi),
\end{equation}
for set of histories $\mathcal{H}$. The following subsection will establish these constructions and prove the claim. Now problem $(P)$ falls in the framework of Section \ref{scn:model} that provides the means to calculate the optimal value and optimal policies as presented in the next chapter.

\subsection{Towards a completely observable problem}
The aim of this subsection is to prove the following:

\begin{theorem}
    Let a risk sensitive POMDP $(\mathcal{S},\mathcal{Y},\mathcal{A},\widehat{P},Q,\widehat{U}),$ with stochastic dynamics as in subsection \ref{original setting} and performance index given in \eqref{fistperformanceindex}. Then, there exist operators $G^{i}:\mathbb{P}(\mathcal{S}) \times \mathcal{A} \times \mathcal{Y} \to \mathbb{R}$, $F^{i}: \mathbb{P}(S) \times A \times Y \to \mathbb{P}(\mathcal{S}),\hspace{4pt} i\in \{1,\dots,i_{\max}\}$, and
    $\eta:\oPiH\rightarrow \PiH,$ such that the following Completely Observable MDP $X_{n}$ with performance index $\Iw_{N}(x_{0},\pi)$ is equivalent the original, i.e. 

$$\mathcal{I}_{N}(x_{0},\eta(\widehat{\pi}))=\widehat{\mathcal{I}_N}(\theta_{0},\widehat{\pi}).$$
The completely observable MDP is defined by the following:
\begin{enumerate}
    \item The state space is $\mathcal{X}=\mathbb{P}(S)^{i_{\max}}\times\mathcal{Y}$ and the set of actions is $\mathcal{A}$
    \item The evolution of the information state $\theta_n^i$ is given by:
    $$\theta_{n+1}^i=F^i(\theta_n,A_n,Y_{n+1})$$ and under the assumption that $Y_n$ are uniformly distributed on the set $\mathcal{Y}$ we arrive at the following transition rule for $x=(\theta^{1},\dots,\theta^{i_{\max}},y)\in \mathcal{X} ,a\in \mathcal{A},$:
    \begin{equation*}
    	P(x'|x;a):=\begin{cases}\frac{1}{|\mathcal{Y}|}, &  \text{if}\hspace{8pt} x'=(F^{1}(\theta^{1},a,y'),\dots,F^{i_{\max}}(\theta^{i_{\max}},a,y'),y'), \\ 0, &\text{otherwise}.\end{cases}
    \end{equation*}
    \item The cost functions are given by: $C: \mathcal{X} \times \mathcal{A}\times \mathcal{X} \to \mathbb{R}$, with:
    $$C^{i}(x,u,x')=G^{i}(\theta^{i},u,y') + \log(|\mathcal Y|^{\frac{1}{\lambda^i}})$$
    
    \item The set of histories is given by $\mathcal{H}_{0}=\mathcal{X},$ and $\mathcal{H}_{n}=\mathcal{H}_{n-1}\times \mathcal{A} \times \mathcal{X}$. A policy $\pi\in\PiH$ takes the form $\pi=(f_0,\dots,f_{n},\dots)$, where $f_{n} : \mathcal{H}_{n} \to \mathcal{A}$.

    \item The optimization problem is governed by the performance index: 
    \begin{equation*}
    	\mathcal{I}_{N}(x_{0},\pi)=\sum_{i=1}^{i_{\max}}w^{i}\text{sign}(\lambda^{i})\mathbb{E}_{x_{0}}^{\pi}\left[e^{\lambda^{i}\left[\sum_{n=0}^{N-1}C^{i}(X_{n},A_{n},X_{n+1})\right]}\right],
    \end{equation*}
\end{enumerate}
We remark that although $\eta$ is not a bijection, it essentially behaves like one in the same way as the one explained in Section \ref{assasin}.
\end{theorem}
% \begin{theorem}
% We have:
% $$\widehat{\mathcal{I}_N}(\theta_{0},\widehat{\pi})=\mathcal{I}_{N}(x_{0},\eta(\widehat{\pi})),$$
% where 
% \begin{equation*}
%     	\mathcal{I}_{N}(x_{0},\pi)=\sum_{i=1}^{i_{\max}}w^{i}\text{sign}(\lambda^{i})\mathbb{E}_{x_{0}}^{\pi}\left[e^{\lambda^{i}\left[\sum_{n=0}^{N-1}C^{i}(X_{n},A_{n},X_{n+1})\right]}\right]
% \end{equation*}
% marks a completely observable optimization problem for appropriate expectation $\mathbb{E}^{{\pi}}_{\theta}$ and cost functions $C^i$.
% \end{theorem}

\begin{proof}
The transformation to the completely observable problem will be done by introducing a sufficient statistic information space realized through information state vectors $\psi$ and $\theta$. The goal is to eliminate the unobservable quantities appearing in $$\widehat{\mathbb{E}}_{\theta_{0}}^{\widehat{\pi}}\left[e^{\lambda^{i}\left[\sum_{n=0}^{N-1}\widehat{C}(S_{n},A_{n})\right]}\right].$$ Towards this goal a new probability measure that eliminates the dependencies that hold the previous measure to the partially observable case is introduced.
Namely, following \cite{cavazos2005} and \cite{fleming1997} there exists a unique  probability measure $\mathbb{P}^{\widehat{\pi}}_{\theta}$  on $\widehat{\mathcal{G}}_{n}$ and its expectation operator $\mathbb{E}^{\widehat{\pi}}_{\theta}$ given by: 

\begin{equation*}
	{\mathbb{P}}^{\widehat{\pi}}_{\theta}\left(s_{0},a_{0},s_{1},y_{1},a_{1},\dots,a_{n-1},s_{n},y_{n}\right):=\theta(s_{0})
	\prod_{k=0}^{n-1}\left(\frac{1}{|\mathcal{Y}|}\widehat{P}\left(s_{k+1}\vert s_{k};\widehat{f}_{k}(h_{k})\right)\right),
\end{equation*}

for some $\theta \in \mathbb{P}(S)$, $\widehat{\pi} \in \oPiH$.
Note that $\theta(s_0)=\theta_0$ in our current set-up.\\

In what follows a relationship between $\widehat{\mathbb{E}}_{\theta_{0}}^{\widehat{\pi}}$ and ${\mathbb{E}}_{\theta_{0}}^{\widehat{\pi}}$ is constructed, in order for the latter to replace the former in the optimization problem.\\

The first important milestone in that direction makes use of the Radon-Nikodym theorem: Namely it can be noted that on the $\sigma$-field $\widehat{\mathcal{G}}_{n},$ the Radon-Nikodyn derivative of $\widehat{P}^{\widehat \pi}_{\theta}$ with respect to $P^{\widehat \pi}_{\theta}$ is given by 

\begin{equation*}
	\frac{\partial\widehat{\mathbb{P}}^{\widehat{\pi}}_{\theta}}{\partial\mathbb{P}^{\widehat{\pi}}_{\theta}}\bigg|_{\widehat{\mathcal{G}}_{n}}=\prod_{k=0}^{n-1}\left(|\mathcal{Y}|Q(Y_{k+1}\vert S_{k+1})\right)=:Z_{n},
\end{equation*}
and therefore
\begin{equation}\label{okay}
	\widehat{\mathbb{E}}_{\theta_{0}}^{\widehat{\pi}}\left[e^{\lambda^{i}\left[\sum_{k=0}^{n}\widehat{C}(S_{k},A_{k})\right]}\right]=\mathbb{E}_{\theta_{0}}^{\widehat{\pi}}\left[e^{\lambda^{i}\left[\sum_{k=0}^{n}\widehat{C}(S_{k},A_{k})\right]}Z_{n}\right].
\end{equation}

After establishing the defining relationship between the two measures for the change of measure in equation \eqref{okay} we now focus on the right part of the equation.
This process will also lead to the construction of an information vector, on which the information space of the transformed MDP will be based, following the next steps:
\medskip
\\
\begin{enumerate}
\itemsep2em 

    \item First, let us define the positive and $\widehat{\mathcal{F}}_{n}$-measurable random variable $\psi^{i}_{n}$, by
\begin{equation*}
	\psi^{i}_{n}(s):=\mathbb{E}_{\theta_{0}}^{\widehat{\pi}}\left[\mathds{1}_{\{S_{n}=s\}}e^{\lambda^{i}\left[\sum_{k=0}^{n}\widehat{C}(S_{k},A_{k})\right]}Z_{n}\Big|\widehat{\mathcal{F}}_{n}\right].
\end{equation*}
$\psi_n$ is a vector $\in \mathbb{R}^{|\mathcal{S}|}$ for each $i \in \{1,...,i_{\max}\}$. Intuitively it can be understood, as \textit{(random) average accumulated cost} up to time $n$ of all outcomes that share the same observations and choices of controls leading to final state $S_n=s$, given the information observable to or controlled by the agent up to this time step.

\item Using $\psi_n$ ,the notation $\int\psi^{i}_{n}=\sum_{s\in\mathcal{S}}\psi_{n}^{i}(s),$ , the linearity of the expectation operator and the tower property of conditional expectation we can then rewrite the right side of \ref{okay}:
$$\mathbb{E}_{\theta_{0}}^{\widehat{\pi}}\left[e^{\lambda^{i}\left[\sum_{k=0}^{n}\widehat{C}(S_{k},A_{k})\right]}Z_{n}\right]=$$

\begin{equation}\label{telescopic}
	\mathbb{E}_{\theta_{0}}^{\widehat{\pi}}\left[\sum_{s\in\mathcal{S}}\mathbb{E}_{\theta_{0}}^{\widehat{\pi}}\left[\mathds{1}_{\{S_{n}=s\}}e^{\lambda^{i}\left[\sum_{k=0}^{n}\widehat{C}(S_{k},A_{k})\right]}Z_{n}\Big|\widehat{\mathcal{F}}_{n}\right]\right]=\mathbb{E}_{\theta_{0}}^{\widehat{\pi}}\left[\int\psi^{i}_n\right]= \mathbb{E}_{\theta_{0}}^{\widehat{\pi}}\left[\int\psi^{i}_{0} \prod^{n}_{k=1}\frac{\int\psi^{i}_k}{\int \psi^{i}_{k-1}}\right].
\end{equation}

\item Then note that setting $\psi^{i}_{0}=\theta_{0}$, $\psi_n^i$ satisfies the following recursion:
\begin{equation}\label{recursion}
	\psi^{i}_{n} = |\mathcal{Y}| M^{i}(A_{n-1},Y_{n})\psi^{i}_{n-1}
\end{equation}
for the matrix $M(a,y) \in \mathbb{R}^{|\mathcal{S}| \times |\mathcal{S}|}$ given by:
\begin{equation*}
	M^{i}(a,y)[s,s']:=\left(e^{\lambda^{i}\widehat{C}(s,a)}\widehat{P}(s'\vert s ; a)Q(y\vert s')\right)^{\intercal}.
\end{equation*}

Inserting the recursion \ref{recursion} in \ref{telescopic} yields:
\begin{equation}\label{inserting}
\mathbb{E}_{\theta_{0}}^{\widehat{\pi}}\left[\int\psi^{i}_n\right]=\mathbb{E}_{\theta_{0}}^{\widehat{\pi}}\left[|\mathcal{Y}|^{n}\int\psi^{i}_{0}\prod^{n}_{k=1}\int  \left[M^{i}(A_{k-1},Y_{k})\frac{\psi^{i}_{k-1}}{\int \psi^{i}_{k-1}}\right]\right].
\end{equation}

Note that the integral $\int  \left[M^{i}(A_{k-1},Y_{k}) \frac{\psi^{i}_{k-1}}{\int \psi^{i}_{k-1}}\right]$ is to be understood in the same way as the previously introduced notation for $\int\psi^{i}_n$. Also $\int\psi^{i}_{0}=1$ per definition.

\item Finally we normalize the information vector $\psi^i_n$ by introducing $\theta^i_n$, so that we arrive at an information state that is an element of  $\mathbb{P}(\mathcal{S})$:

\begin{equation*}
\theta^{i}_{n}=\frac{\psi^{i}_{n}}{\int \psi^{i}_{n}},
\end{equation*} 

Replacing $\psi_n^i$ with  $\theta_n^i$ in \ref{inserting} yields:
\begin{equation}\label{exp}
\mathbb{E}_{\theta_{0}}^{\widehat{\pi}}\left[\int\psi^{i}_n\right]\!{=}\mathbb{E}_{\theta_{0}}^{\widehat{\pi}}\left[\!|\mathcal{Y}|^{n}\prod^{n}_{k=1}\int  M^{i}(A_{k-1},Y_{k})\theta^{i}_{k-1}\!\right]{=}
   \mathbb{E}_{\theta_{0}}^{\widehat{\pi}}\left[\!e^{\lambda^{i}\left(\sum^{n}_{k=1}\left(\frac{1}{\lambda^{i}}\log\left(\int [M^{i}(A_{k-1},Y_{k})\theta^{i}_{k-1}]\right)+\log\mathcal|{Y}|^{\frac{1}{\lambda^{i}}}\!\right)\!\right)}\!\right], 
\end{equation}

where in the last step we have used the properties of the exponential and logarithmic functions in order to rewrite the operation in a more suitable form.
\end{enumerate}

\bigskip

In steps 1-4 we have thus achieved our goal of rewriting \ref{okay} only using quantities known to the agent. Namely we have shown:
\begin{equation*}
    	\widehat{\mathbb{E}}_{\theta_{0}}^{\widehat{\pi}}\left[e^{\lambda^{i}\left[\sum_{k=0}^{n}\widehat{R}(S_{k},A_{k})\right]}\right] = \mathbb{E}_{\theta_{0}}^{\widehat{\pi}}\left[e^{\lambda^{i}\left(\sum^{n}_{k=1}\left(\frac{1}{\lambda^{i}}\log\left(\int [M^{i}(A_{k-1},Y_{k})\theta^{i}_{k-1}]\right)+\log\mathcal|{Y}|^{\frac{1}{\lambda^{i}}}\right)\right)}\right] 
\end{equation*}
In addition we have constructed an information state $\theta_n^i \in \mathbb{P}(\mathcal{S})$ for the transformed MDP.

\bigskip
% As a last step we still need to introduce some operators that will simplify and ease the notation in the transformed model and properly define the new optimization problem.

As a last step we introduce the notation for the new optimization problem that leads to the direct claim of this theorem. First consider $G:\mathbb{P}(\mathcal{S}) \times \mathcal{A} \times \mathcal{Y} \to \mathbb{R}$, given by:

\begin{equation*}
  G^{i}\left(\theta^{i},a,y\right):=\frac{1}{\lambda^{i}}\log\left(\int M^{i}(u,y)\theta^{i}\right).  
\end{equation*}
This lets us rewrite \ref{exp}:
\begin{equation*}
    \mathbb{E}_{\theta_{0}}^{\widehat{\pi}}\left[e^{\lambda^{i}\left(\sum^{n}_{k=1}\left(\frac{1}{\lambda^{i}}\log\left(\int [M^{i}(A_{k-1},Y_{k})\theta^{i}_{k-1}]\right)+\log\mathcal|{Y}|^{\frac{1}{\lambda^i}}\right)\right)}\right] =     
    \mathbb{E}_{\theta_{0}}^{\widehat{\pi}}\left[e^{\lambda^{i}\left(\sum^{n}_{k=1}\left(G(\theta_k^i, A_{k-1}, Y_k) + \log\mathcal|{Y}|^{\frac{1}{\lambda^i}}\right)\right)} \right]
\end{equation*}

\bigskip

Furthermore we use use \ref{recursion} on $\theta_n^i$:
\begin{equation*}
    \theta^{i}_{n}=\frac{\psi^{i}_{n}}{\int \psi^{i}_{n}}= \frac{|\mathcal{Y}| M^{i}(A_{n-1},Y_{n})\psi^{i}_{n-1}}{\int|\mathcal{Y}| M^{i}(A_{n-1},Y_{n})\psi^{i}_{n-1}} =  \frac{ M^{i}(A_{n-1},Y_{n})\theta^{i}_{n-1}}{\int M^{i}(A_{n-1},Y_{n})\theta^{i}_{n-1}}
\end{equation*}
We this recursion we can now write $F: \mathbb{P}(S) \times A \times Y \to \mathbb{P}(\mathcal{S})$ as a forward operation for the information state:
\begin{equation*}
    F^{i}\left(\theta^{i},a,y\right):=\frac{M^{i}(a,y) \theta^{i}}{\int  M^{i}(a,y)\theta^{i}}.
\end{equation*}
For the reformulation of the MDP we can now use the cost functions: $C: \mathcal{X} \times \mathcal{A}\times \mathcal{X} \to \mathbb{R}$, with:
    $$C^{i}(x,u,x')=G^{i}(\theta^{i},u,y') + \log(|\mathcal Y|^{\frac{1}{\lambda^i}})$$
Furthermore, for the policies we have  $\eta:\oPiH\rightarrow \PiH$ such that $(f_{n})_{n\in \{0,\dots,N-1\}}=\eta((\widehat{f}_{n})_{n\in \{0,\dots,N-1\}})$ satisfies
    \begin{equation*} 
    	f_n(x_{0},a_{0},\dots,x_{n-1},a_{n-1},x_{n})=\widehat{f}_{n}(\theta_{0},a_{0},\dots,y_{n-1},a_{n-1},y_{n}).
    \end{equation*}
\end{proof}
%%%%%%%%%%%
%%%%%%%%%%
%%%%%%%%%%
\bigskip
\bigskip

\section{MDP with Multivariate Utility Function}

In this section, we describe a model for  risk sensitive multi-objective sequential decision making on a Borel state and action space with multiple costs and a multivariate utility function. The performance index is the expected multivariate utility, where each variable corresponds to a different running cost. As a generalization to the classical MDP model, we allow for the cost to depend on the subsequent state in addition to the current state-action pair. We thereby follow and extend ideas from \citet{bauerle2014} and \citet{hernandez1996b}.

\subsection{Notation and Assumptions} \label{scn:model}

Throughout this section, we assume that an $N$-step Markov Decision Process is given by a Borel state space $\X$, a Borel action space $\A$, a Borel set $\D \sse \X \times \A$, and a regular conditional distribution $\mathrm{P}$ from $\X \times \D$ to $[0,1]$. Given the current state $x \in \X$, we assume that an action $a \in D(x)$ may be chosen, where $D(x) := \{ a \in \A \mid (x,a) \in \D \}$ is the set of feasible actions. The transition probability to the next state is then given by the distribution $\mathrm{P}(\cdot | x ; a)$, according to the chosen action. The set of histories from up to time $n$ is defined by
\begin{equation*}
	\H_{0} := \X, \hspace{32pt}
	\H_{n} := \H_{n-1} \times \A \times \X, \quad n\in\mathbb{N}
\end{equation*}
and $h_{n} = (x_{0},a_{0},\dots,x_{n}) \in \H_{n}$ is a historical outcome  up to time $n.$

\begin{definition} The set of (history-dependent) policies is defined by
	$$\PiH:= \left\{ \pi = (f_{0},f_{1},\dots) \middlemid f_{n} : \H_{n} \to \mathcal{A},\quad \forall h_{n} \in \H_{n} : f_{n}(h_{n}) \in D(x_n),\quad n\in\mathbb{N} \right\}.$$ 
	Similarly, the set of Markovian policies is defined by
	$$\PiX := \left\{ \pi = (g_{0},g_{1},\dots) \middlemid g_{n} : \X \to \mathcal{A}, \quad\forall x \in \X : g_{n}(x) \in D(x),\quad n\in\mathbb{N} \right\}.$$
\end{definition}

Given an initial state $x \in \X$ and a history-dependent policy $\pi = (f_1, f_2, \dots) \in \PiCh$, due to the Ionescu-Tulcea theorem, there exists a probability measure $\P_{x}^\pi$ on $\H_{\infty}$ and two stochastic processes $(X_n)_{n}$, $(A_n)_{n}$ such that
$$\P_{x}^\pi(X_{0} \in B) = \delta_x(B),\ \P_{x}^\pi(X_{n+1} \in B \mid H_n, A_n) = \mathrm P(X_{n+1} \in B \mid X_n, A_n)$$
and
$$A_n = f_n(H_n)$$
for all Borel sets $B \sse \X$. Canonically, $H_n$, $X_n$, $A_n$ are the history, state and action at time $n$.  By $\E_x^\pi$  we denote the expectation operator corresponding to $\P_x^\pi.$ For more details of this construction, we refer to \citet{bauerle2014}.\\

Throughout the whole section, we have the following standing assumptions:

\begin{assumption}\label{assumption} ~
	\begin{enumerate}
		\item The utility function $\Uw: \mathbb{R}^{i_{\max}} \to \R$ is continuous, and it exists $\hspace{2pt}0\leq i_{\tau}\leq i_{\max}\hspace{2pt},$ such that $\Uw$ is component-wise increasing in $\{i<i_{\tau}\}$ and  component-wise decreasing in $\{i>i_{\tau}\}.$
		\item  $\|\mathcal{U}\|_{\infty}<\infty.$ 
		\item The sets $D(x), x \in \X$ are compact.
		\item The map $x \mapsto D(x)$ is upper semi-continuous, i.e. if $x_n \to x \in X$ and $a_n \in D(x_n)$, then $(a_n)$ has an accumulation point in $D(x)$.
		\item The maps $(x,a,x') \mapsto C^{i}(x,a,x'),\ i=1,\dots,\imax,$ are continuous, and it holds $\underline{c} \leq C^i(\cdot, \cdot, \cdot) \leq \bar c$ for some fixed $\underline c, \bar c \in \mathbb{R}$.
		\item $\mathrm{P}$ is weakly continuous.
	\end{enumerate}
\end{assumption}

\begin{remark}
Due to uniform boundedness of the cost functions $C^{i},$ and the fact that we work on finite horizon problems with $N<\infty$ or infinite horizon problems with discount, we can observe that assumption \textit{2} can be removed without any loss of generality.
\end{remark}

For notational convenience, we define the vector valued function $\C: \X \times \A \times \X \to \R^\imax$ by \begin{equation*}
\C(x,a,x') := \left(C^1(x,a,x'), \dots, C^\imax(x,a,x')\right).\end{equation*}

\subsection{Finite Horizon Problems}
\subsubsection{Performance Index}
After we have set the stage for Markov Decision Processes and their policies, we can now define a performance index that is the expected utility of several running costs.

\begin{definition}
	Denote by $N$ the number of steps of the MDP. We define the \emph{total cost} $\Iw_N(x,\pi)$ given an initial state $x \in \mathcal{X},$ and a history dependent policy $\pi \in \PiCh$ by
	$$\Iw_{N}(x,\pi):=\E_{x}^{\pi}\left[\Uw\left(\sum_{n=0}^{N-1}\C(X_{n},A_{n},X_{n+1})\right)\right],$$
	and the corresponding value function by
	\begin{equation} \label{PTotal} \tag{$P$}
	\Vw_{N}(x) := \inf_{\pi\in\PiH}\Iw_{N}(x,\pi).
	\end{equation}		
\end{definition}

\subsubsection{Augmented problem}
The aim of what follows is to determine $\mathcal{V}_N,$ and optimal policies in \eqref{PTotal}. To this end, we augment the state space of the MDP to $\X \times \R^\imax$. The second component models the so-far accumulated cost of the advancing MDP. In particular, $\widetilde{X}_n:=(X_{n},\mathcal{R}_{n}) \in \X \times \R^\imax$ taking the value $(x, \d) = (x, r^1, \dots, r^\imax)$ implies that the MDP has advanced to state $x$ and accumulated a cost amounting to $r^i$ in the $i$-th objective after the first $n$ steps. In order to define transition probabilities of the augmented problem, we introduce the notion of a pushforward measure.

\begin{definition}
	\label{push} Given measurable spaces $(\mathcal{S},\mathcal{F})$, $(\widetilde{\mathcal{S}},\widetilde{\mathcal{F}})$, a measurable mapping $\mathcal{T}:\mathcal{S}\rightarrow\widetilde{\mathcal{S}}$ and a
	measure $\mu:\mathcal{F}\rightarrow\lbrack0,\infty]$, the pushforward of $\mu$
	is the measure induced on $(\widetilde{\mathcal{S}},\widetilde{\mathcal{F}})$ by $\mu$ under
	$\mathcal{T}$, i.e., the measure $\mathcal{T}_{\#}\mu:\widetilde{\mathcal{F}}\rightarrow
	\lbrack0,\infty]$ is given by%
	\[
	(\mathcal{T}_{\#}\mu)(B)=\mu\left(  \mathcal{T}^{-1}(B)\right)  \mbox{ for }B\in\widetilde
	{\mathcal{F}}.
	\]
\end{definition}

In particular, if a function $f$ is $\wt{\mathcal{F}}$-measurable and $\mathcal{T}_\# \mu$-integrable, and $f \circ \mathcal{T}$ is $\mu$-integrable, then
$$\int f\ d\mathcal{T}_\# \mu = \int f \circ \mathcal{T}\ d\mu.$$

Now, we define the transition kernel $\widetilde P$ of the augmented problem by
\begin{equation}
	\widetilde{\mathrm{P}}(\cdot|\widetilde{x};a)
	=\widetilde{\mathrm{P}}(\cdot|(x,\d);a)
	=(\mathcal{T}_{(x,\d)})_{\#}\mathrm{P}(\cdot|x,a),
\end{equation}
where 
\begin{equation}
\mathcal{T}_{(x,\d)}(x')=(x',\mathbf{C}(x,a,x')+\d).
\end{equation}
If $\mathcal{X}$ is finite, this leads to
\begin{equation}
	\widetilde{\mathrm{P}}(\widetilde{x}'|\widetilde{x};a)=
	\begin{cases}
		\mathrm{P}(x'|x;a), & \text{if } \widetilde{x}=(x,\d), \widetilde{x}' = (x',\d+\C(x,a,x')), \\
		0, & \text{otherwise}.
\end{cases}
\end{equation}
The histories for the augmented MDP are given by
\begin{equation*}
\widetilde{\mathcal{H}}_{0} := \X \times \R^\imax, \hspace{16pt}
\widetilde{\mathcal{H}}_{n} := \widetilde{\mathcal{H}}_{n-1} \times \D \times \left( \X \times \R^\imax \right), \quad n\in\mathbb{N}.
\end{equation*}
The definition of history-dependent policies $\widetilde \pi \in \wPiH$, Markovian policies $\widetilde \pi \in \wPiX$, and the corresponding decision rules are changed accordingly. \\

Similar to the previous section, there exist a probability measure $\wt \P_{x}^\pi$ on $\wt \H_{\infty}$ and a coupled stochastic process $(\wt X)_{n\in\N}$ with $\wt X_n= (X_n, R_n)$, and a stochastic process $(A_n)_{n \in \N}$, such that
% two stochastic processes $(\wt X_n)_{n=0,\dots,N \color{red} n \in \N}$, $(\wt D_n)_{n=0,\dots,N-1 \color{red} n \in \N}$ such that
$$
	\wt \P_{\wt x}^{\wt \pi} (\wt X_{0} \in B) = \delta_{\wt x}(B), \quad
	\wt \P_{\wt x}^{\wt \pi} (\wt X_{n+1} \in B \mid \wt H_n, A_n) = \widetilde{\mathrm{P}}(\wt X_{n+1} \in B \mid \wt X_n, A_n)$$
and
$$A_n = \wt f_n(\wt H_n)$$
for all Borel sets $B \sse \X$, and $H_n$, $X_n$, $A_n$ are the history, state and action at time $n$, given an initial state $\wt x \in \X \times \R^\imax$ and a history-dependent policy $\wt \pi \in \wPiH$. By induction, it is easy to prove that $\mathcal{R}_{n}=\sum_{k=0}^{n-1}\C(X_{k},A_{k},X_{k+1})+\d,$
$\widetilde{\mathbb{P}}^{\mathcal{\widetilde{\pi}}}_{\widetilde{x}}$-almost surely.

\begin{definition}
	Denote by $N\in\mathbb{N}$ the number of steps of the MDP. For $n=1,\dots,N$, we define the \emph{total cost} $\wt \Iw_n((x,\d), \widetilde\pi)$ for the augmented problem, given the initial state $x\in \X$, initial cost $\d \in \R^\imax$, and policy $\widetilde\pi \in \widetilde{\Pi}$ by
	\begin{equation}\label{ara}
		\widetilde{\Iw}_n(\widetilde{x},\widetilde{\pi})=\widetilde{\Iw}_n((x,\boldsymbol{r}),\widetilde{\pi}):=
		\widetilde{\E}^{\widetilde{\pi}}_{\widetilde{x}} \left[
			\Uw \left(
				\sum_{k=0}^{n-1}\C(X_{k},A_{k},X_{k+1})+\d
			\right)
		\right]=\widetilde{\E}_{\widetilde{x}}^{\widetilde{\pi}} \left[
			\Uw \left(R_n) \right)
		\right],
	\end{equation}
	and the corresponding value function by
	\begin{equation} \label{PCostToGo} \tag{$\widetilde P$}
	\widetilde{\mathcal{V}}_{N}(\widetilde{x}) := \inf_{\widetilde{\pi}\in\wPiH} \widetilde{\Iw}_{N}(\widetilde{x},\widetilde{\pi}).
	\end{equation}
\end{definition}

\subsubsection{Policy bijection} \label{assasin}
For the sequel, let $ x\in\mathcal{X},$ and $\widetilde{x}=(x,\mathbf{0}).$ Note that policies $\pi = (f_{0},f_{1},\dots)\in\PiCh$ of the original problem consist of functions $f_{n}$ that are defined on $\mathcal{H}_{n}$, and policies  $\widetilde \pi \in \wPiH$ consist of functions $\widetilde{f}_{n}$ defined on $\widetilde{\mathcal{H}}_{n}.$ Therefore there is no simple bijectional correspondence between the two sets. However, the set of histories 

\begin{equation}
\widetilde{\mathcal{H}}^{-}_{n}=\{\widetilde{h}_{n}\in\widetilde{\mathcal{H}}_{n}|\left( \exists k\in\{1,\dots,n-1\} : \d_{k}\neq \d_{k-1}+\C(x_{k},a_{k},x_{k+1})\right) \vee (\d_{0}\neq\mathbf{0})\},\end{equation}
is not accessible in the sense that these histories cannot occur. In a more rigorous manner, we have that $\widetilde{\mathbb{P}}^{\widetilde{\pi}}_{\tilde{x}}(\widetilde{\mathcal{H}}^{-}_{n})=0,$ for all $\widetilde{\pi}\in\wPiH.$
For every policy $\widetilde{\pi}\in \wPiH,$ we may define a new ``reduced'' policy $\widetilde{\pi}^{\red}$ by
\begin{equation}
 \widetilde{f}^{\red}_{n}(\widetilde{h}_{n})=\begin{cases}
 a^{\red}(x_{n}), &  \text{if } \widetilde{h}_{n}\in\widetilde{\mathcal{H}}^{-}_{n},\\
  \widetilde{f}_{n}(\widetilde{h}_{n}), & \text{otherwise},
 \end{cases}
\end{equation}
where $a^{\red}$ can be any arbitrary but fixed point in $\D(x_{n}).$
Then for the set $\wPiH^{\red}=\{\widetilde{\pi}^{\red}\in\wPiH:\widetilde{\pi}\in\wPiH\},$ we can define a bijection to $\PiH.$
To do so, for $\pi = (f_{0},f_{1},\dots)\in\PiCh,$ we define $\widetilde{\pi}^{\red}=( \widetilde{f}^{\red}_{0},\widetilde{f}^{\red}_{1},\dots)\in \wPiH^{\red},$ by

 \begin{equation}\label{polpol}
 \begin{split}
 \widetilde{f}^{\red}_{n}\left(\left(x_0,0\right),a_{0},\dots,\left(x_{n-1},\sum_{i=0}^{n-1}\mathbf{C}(x_{i},a_{i},x_{i+1})\right),a_{n-1},\left(x_{n},\sum_{i=0}^{n}\mathbf{C}(x_{i},a_{i},x_{i+1})\right)\right)=\\
 \begin{cases}
a^{\red}, &  \text{if } \widetilde{h}_{n}\in\widetilde{\mathcal{H}}^{-}_{n}, \\
f_{n}(x_{0},a_{0},\dots,x_{n-1},a_{n-1},x_{n}), & \text{otherwise}.
\end{cases}
\end{split}
 \end{equation} 

It is easy to see that the value function of (\ref{PTotal}) coincides with the value function of (\ref{PCostToGo}) with $\d=\mathbf{0}$, i.e.
$$\mathcal{V}_N(x) =\inf_{\pi\in\PiH}\Iw_{N}(x,\pi)=  \inf_{\widetilde{\pi}\in\wPiH^{\red}} \widetilde{\Iw}_{N}((x,\mathbf{0}),\widetilde{\pi})=  \inf_{\widetilde{\pi}\in\wPiH} \widetilde{\Iw}_{N}((x,\mathbf{0}),\widetilde{\pi})=\widetilde{\mathcal{V}}_{N}((x, \mathbf{0})).$$ 

The next step is to derive a Bellman-style equation for the augmented problem (\ref{PCostToGo}). It can be shown that the minimizer of (\ref{PCostToGo}) is a Markovian policy.

\subsubsection{Bellman operator and first theorem}

First, for a fixed $m\in\mathbb{R},$ we define the set
\begin{align*}
& \Delta := \big\{ v : \X \times \R^\imax \to \R ~\big\vert~ v \hspace{6pt} \text{is lower semi-continuous,} \hspace{6pt} v(x,\cdot)\hspace{6pt} \text{is continuous, }\hspace{6pt} \|v\|_{\infty}<\infty, \hspace{6pt} \inf_{x,d}\{v(x,\d)\} \geq m, \\
&  \text{and for all}\hspace{6pt} x \in \X \hspace{6pt}\text{is component-wise increasing (decreasing) in}\hspace{6pt} \{i<i_{\tau}\}\hspace{6pt}  \text{(in $\{i>i_{\tau}\}$}) \big\}.
\end{align*}

For $v \in \Delta$ and a Markovian decision rule $\widetilde g \in \wPiX$, we define the operators
\begin{equation*}
	\begin{split}
		T_{\g}[v](\widetilde{x})
		&= T_{\g}[v](x,\d)
		= \int v\left(\widetilde{x}'\right)\widetilde{\mathrm{P}}\left(d\widetilde{x}' \mid \widetilde{x},\g(\widetilde{x})\right)
		= \int v\left((x',r')\right)(\mathcal{T}_{(x,\d)})_{\#}\mathrm{P}\left(dx' \mid x,\g(x,\d)\right) \\
		& =\int v\left(x',\C(x,\g(x,\d),x')+\d\right)\mathrm{P}(dx' \mid x,\g(x,\d)),
	\end{split}
\end{equation*}
and
\begin{equation*}
T[v](x,\d)=\inf_{a\in D (x)}\int v(x',\C(x,a,x')+\d)P(dx' \mid x,a),
\end{equation*}
whenever the integrals exist. $T$ is called the \emph{minimal cost operator}. We say that a Markovian decision rule $\widetilde g$ is a \emph{minimizer} of $v$ if $T_{\widetilde g}[v] = T[v]$. In this situation, $\widetilde g(x,\d)$ is a minimizer of
$$\D(x) \ni a \mapsto \int v\left(x',\C(x,a,x')+\d\right)P(dx' \mid x,a)$$
for every $(x,\d) \in \X \times \R^\imax$. We may now state the main result of this section:

\begin{theorem} \label{thm:SumsOfCostThm}
	Let $\wt\Vw_{0}(x,\d) := \Uw (\d)$. Then, the following holds:
	\begin{enumerate}[label=\alph*)]
		\item For any Markovian policy $\widetilde \pi=(\g_0, \g_1, \dots) \in \wPiX$, we have the cost iteration
		\begin{equation*}
			\wt \Iw_{n}((x,\d),\widetilde{\pi}) = T_{\wt g_0}[\dots [T_{\wt g_{n-1}}[\wt\Vw_{0}]]](x,\d)
		\end{equation*}
		for all $n=1,\dots,N.$
		
		\item If an optimal policy exists it is Markovian, i.e.
		$$\inf_{\widetilde\pi \in \wPiH} \Iw_{N}(\widetilde{x},\widetilde\pi) = \inf_{\widetilde\pi \in \wPiX} \Iw_{N}(\widetilde{x},\widetilde\pi).$$
		
		\item The operator $T: \Delta \to \Delta$ is well-defined, and for every $v \in \Delta$, there exists a minimizer of $T[v]$.
		
		\item We get the Bellman-style equation
		$$\wt\Vw_{n}(x,\d)= T[\wt\Vw_{n-1}](x,\d)
		= \inf_{a\in D(x)} \int \wt\Vw_{n-1}(x',\C(x,a,x')+\d)P(dx' \mid x,a)$$
		for all $n=1,\dots,N$.
		
		\item If $\widetilde g^*_{n}$ is a minimizer of $\wt \Vw_{n-1}$ for $n=1,\dots,N$, then the history-dependent policy $\pi^* = (f^*_0,\dots, f^*_{N-1})$, defined by
		$$f_n^*(h_{n}):=
		\begin{dcases}
		\widetilde g^*_{N}(x_0, 0) & \text{if } n=0, \\
		\widetilde g^*_{N-n} \left(x_n,\sum_{k=0}^{n-1} \C(x_k,a_k,x_{k+1})\right) & \text{otherwise},
		\end{dcases}
		$$
		is an optimal policy for problem (\ref{PTotal}).\medskip
	\end{enumerate}
\end{theorem}

For the proof of Theorem \ref{thm:SumsOfCostThm}, we need the following lemma:\medskip

\begin{lemma} \label{thm:minexists}
	Let $v:\X \times \R^\imax \to \R$ be bounded and lower semi-continuous. Suppose
	\begin{enumerate}
		\item $D(x)$ is compact,
		\item $x \mapsto D(x)$ is upper semi-continuous,
		\item $(x,\d,\g,x') \mapsto v(x', \C(x,\g(x,\d),x')+\d)$ is lower semi-continuous. 
	\end{enumerate} Then, $Tv$ is is lower semi-continuous and there exists a minimizer $\g^*$ such that  $T_{\g^*}v = Tv$.
\end{lemma}

\begin{proof}
	By Lemma 17.11 in \citet{hinderer1970}, $(x,\d,\g) \mapsto T_{\g}v(x,\d)$ is lower semi-continuous. The claim then follows from a similar argument to Proposition 2.4.3 in \citet{bauerle2011}.  \qed
\end{proof}

\begin{proof}[Proof of Theorem \ref{thm:SumsOfCostThm}]
	The proof is similar to Theorem 2.3.4 and Theorem 2.3.8 in \citet{bauerle2011} with a different state space, see also \citet{bauerle2014}.
	%\begin{enumerate}[wide, labelwidth=!, labelindent=0pt, label=\textit{{ad }\alph*)}]
	\begin{enumerate}[label=\textit{{ad }\alph*)}]
		\item
		An easy calculation shows that
		\begin{align*}
		\wt \Iw_1((x,\d),\widetilde{\pi})
		&= \mathbb{E}^{\widetilde{\pi}}_x \Big[
		\Uw \left(
		\C(X_0, A_0, X_1) + \d \right)
		\Big] \\
		&= \int \Uw \left(
		\C(x, \wt g _0(x,\d), x') + \d
		\right) P(dx' \mid x, \wt g_1(x,\d)) = T_{\wt g_1} [\wt \Vw_0](x,\d).
		\end{align*}
		Now, let $\wt \pi^+ = (\wt g_2, \dots)$. For $n = 2,\dots,N$, we get
		\begin{align*}
		&\wt \Iw_n((x,\d),\widetilde{\pi})
		= \mathbb E^{\wt \pi}_x \left[
		\Uw \left(
		\sum_{k=0}^{n-1} \C(X_k, A_k, X_{k+1}) + \d
		\right)
		\right] \\
		&= \int \E^{\wt\pi^+}_{x'} \left[
		\Uw \left(
		\sum_{k=0}^{n-2} \C(X_k, A_k, X_{k+1}) + \d + \C(x, \wt g_{1}(x, \d), x')
		\right)
		\right] P(dx' \mid x,\wt g_0(x,\d)) \\
		&= \int \wt \Iw((x', \d), \wt \pi^+) P(dx' \mid x, \wt g_1(x,\d)) = T_{\wt g_1}[\wt \Iw_{n-1}(\cdot, \wt \pi^+)](\wt x)
		= T_{\g_{1}}[\dots [T_{\g_{n-1}}[\wt\Vw_0]]](x,\d).
		\end{align*}
		The claim follows then by induction.
		
		\item This follows from Theorem 2.2.3 in \citet{bauerle2011}.
		
		\item
		Note that every $v \in \Delta$ is bounded from below by $m$.	By our assumptions, we get that $(x,\d,\g,x') \mapsto v(x', \C(x,\g(x,\d),x')+\d)$ is lower semi-continuous, and bounded from below, i.e. we are in the setting of Lemma \ref{thm:minexists}. Thus, $T[v]$ is lower semi-continuous and there exists a minimizer $\g^*$ such that  $T_{\g^*}[v] = T[v]$.
		
		For fixed $x \in \X$, and $a \in \D(x)$, the map $\d \mapsto \int v(x', \C(x,a,x')+\d) P(dx' \mid x,a)$ has the same monotonicities with respect to $r_{i}'s$ as $v$ and it is continuous for every $a \in D(x)$. The continuity can be proven with the dominated convergence theorem since $\|v\|_{\infty}<\infty$. Therefore, the infimum of these maps over all $a \in D(x)$ is  upper semi-continuous in $\d$. With this, we have shown that $Tv(x,\cdot)$ is upper and lower semi-continuous, and therefore continuous, and respects the same monotonicities as $v$ for all $x \in \X$. Because $v(x,\cdot) \geq m$, we have $T[v](x,\cdot) \geq m$. The boundness assumption follows from the definition of $T$ and the corresponding property of $v$. We have then shown that $T: \Delta \to \Delta$ is well-defined.
		
		\item
		Let $\g^*_{n}$ be a minimizer of $\Vw_{n-1}$ for $n=1,\dots,N$ and denote by $\pi^* = (\g^*_1,\dots,\g^*_{N})$ the associated policy.
		
		For $n = 1$, we get that
		\begin{equation*}
		\begin{split}
		\wt \Vw_1(x,\d)
		= \inf_{\widetilde{\pi}\in  \wPiX}\E_x^{\wt\pi} \Big[
		\Uw \left(
		\C(X_0, A_0, X_1) + \d
		\right)
		\Big] {=} \inf_{a \in D(x)} \int \Uw(\C(x,a,x') + \d) P(dx' \mid x,a) = T[\wt\Vw_0](x,\d),
		\end{split}
		\end{equation*}
		and obviously, $\wt \Vw_1(x,\d) = \wt \Iw_1((x,\d),\widetilde{\pi}^*)$.
		Now, assume that $\wt \Iw_n((x,\d),\widetilde{\pi}^*) = \wt \Vw_n(x,\d)$ for a fixed $n \in \{1,\dots,N\}$. Then,
		\begin{align*}
		\wt \Iw_{n+1}((x,\d),\wt \pi^*)
		&= T_{\wt g^*_1}[\wt \Iw_n(\cdot,(\wt \pi^*)^+)](x,\d) & \text{using (a),} \\
		&= T_{\wt g^*_1}[\wt \Vw_n](x,\d) & \text{by the induction hypothesis,} \\
		&= T[\wt \Vw_n](x,\d) & \text{by definition of $\g^*_1$.}
		\end{align*}
		By taking the infimum, we get
		\begin{equation}\label{eqn:inf1}
		\inf_{\wt \pi \in \wPiX} \wt \Iw_{n+1}((x,\d),\wt\pi) = \wt \Vw_{n+1}(x,\d) \leq \wt \Iw_{n+1}((x,\d),\wt \pi^*) = T[\wt \Vw_{n}](x,\d).
		\end{equation}
		On the other hand, with an arbitrary policy $\widetilde \pi = (\wt g_1,\dots,\wt g_{N})$,
		\begin{align*}
		\wt \Iw_{n+1}((x,\d),\wt\pi)
		&=T_{\wt g_1}[\wt \Iw_n(\cdot,\wt \pi^+)](x,\d) & \text{using (a),} \\
		&\geq T_{\wt g_1}[\wt \Vw_n](x,\d) & \text{by the monotonicity of $T$,} \\
		&\geq T [\wt \Vw_n](x,\d) & \text{by taking the infimum.}
		\end{align*}
		By taking the infimum, we get
		\begin{equation} \label{eqn:inf2}
		\inf_{\widetilde{\pi}\in \wPiX} \wt \Iw_{n+1}((x,\d),\widetilde{\pi}) = \wt \Vw_{n+1}(x,\d) \geq T[\wt \Vw_n](x,\d).
		\end{equation}
		From (\ref{eqn:inf1}) and (\ref{eqn:inf2}), it follows by induction that
		$$
		\wt \Vw_{n}(x,\d)
		= T[\wt \Vw_{n-1}](x,\d)
		= \wt \Iw_n((x,\d),\wt \pi^*)
		$$
		for all $n = 1,\dots,N$.		
		
		\item 
		Consider the Markovian policy $\widetilde{\pi}^* = (\widetilde g^*_1,\dots,\widetilde g^*_{N})$ as defined in \textit{(d)}. We have just shown that $\wt \Vw_N(x,\d) = \wt \Iw_N((x,\d),\widetilde\pi^*)$, i.e. $\widetilde\pi^*$ is a minimizer of (\ref{PCostToGo}), and therefore an optimal policy for the $N$-step MDP with states in $X \times \R^\imax$. The claim follows by \eqref{polpol} in Section \ref{assasin} where the policy bijection is explored. \qed
	\end{enumerate}
\end{proof}

\subsection{Discounted finite horizon problems}

We now consider finite horizon problems with a discount vector $\boldsymbol{\beta} \in (0,1)^{\imax}$, and prove the corresponding analogon to Theorem \ref{thm:SumsOfCostThm}. The techniques used are similar to those from \citet{bauerle2014}, where they are applied to \textit{univariate} utility functions. Similar to the previous setting, we define the \emph{total cost} $\Iw_N(x,\pi)$, given an initial state $x \in \mathcal{X},$ and a history dependent policy $\pi \in \PiCh$, by

\begin{equation}\label{dot}\Iw_{N}(x,\pi):=\E_{x}^{\pi}\left[\Uw\left(\sum_{n=0}^{N-1} \boldsymbol{\beta}^{n}\cdot \C(X_{n},A_{n},X_{n+1})\right)\right],\end{equation}

and the corresponding value function, by
\begin{equation}\tag{$P$}
\Vw_{N}(x) := \inf_{\pi\in\PiH}\Iw_{N}(x,\pi).
\end{equation}
We remark that the dot product appearing in  \eqref{dot} is a componentwise product, i.e. $$\mathbf{a}\cdot\mathbf{b} =(a_{1}b_{1},\dots,a_{n}b_{n})\hspace{32pt}\text{for } \mathbf{a}=(a_{1},\dots,a_{n}), \mathbf{b}=(b_{1},\dots,b_{n}).$$

\subsubsection{Augmented problem}
Again, we consider an augmented state space $\X \times \R^\imax \times (0,1)^{i_{\max}}$, where the new components keep track of the decreasing discount factor. Policy augmentation is done similar to the previous section. The new transition kernel $\widetilde P$ of the augmented problem is given by
\begin{equation}
\widetilde{\mathrm{P}}(\cdot|\widetilde{x};a)
=\widetilde{\mathrm{P}}(\cdot|(x,\d,\z);a)
=(\mathcal{T}_{(x,\d,\z)})_{\#}\mathrm{P}(\cdot|x,a),
\end{equation}
where 
\begin{equation}
\mathcal{T}_{(x,\d,\z)}(x')=(x',\z\cdot \mathbf{C}(x,a,x')+\d, \z\cdot \bb).
\end{equation}
On the augmented state space, given an initial state $x \in \X$, initial cost $\d \in \R^\imax$, initial discount rates $\z \in (0,1)^{i_{max}}$, and a policy $\widetilde\pi \in \wPiH$ the total cost $\wt \Iw_n((x,\d, z), \widetilde\pi)$ for $n=1,...,N$ is given by
\begin{equation}
	\wt\Iw_{n}(\wt x, \wt \pi) = \widetilde{\Iw}_{n}((x,\d,z),\widetilde{\pi}):=
	\widetilde{\E}_{\wt x}^{\widetilde{\pi}} \left[
	\Uw \left(
	\z\cdot \sum_{k=0}^{n-1} \bb^{k} \cdot\C(X_{k},A_{k},X_{k+1})+\d
	\right)
	\right].
\end{equation}
The corresponding value function is
\begin{equation}\tag{$\widetilde P$}
	\widetilde{\mathcal{V}}_{n}(\widetilde{x}) := \inf_{\widetilde{\pi}\in\wPiH} \widetilde{\Iw}_{n}(\widetilde{x},\widetilde{\pi}).
\end{equation}
\subsubsection{Bellman operator and second theorem}

Let
\begin{align*}
	\Delta := \big\{ v : \X \times \R^\imax &\times (0,1)^\imax  \to \R ~\big\vert~
	 \text{$v$ is lower semi-continuous,} \\
	& \text{$v(x,\cdot,\cdot)$ is continuous, and componentwise increasing for all $x \in \X$}, \\
	& \text{$v(x,\d, \z) \geq \Uw(\d)$ for all $(x,\d, \z) \in \X \times \R^\imax \times (0,1)$} \big\}.
\end{align*}

For $v \in \Delta$ and a Markovian decision rule $\widetilde g$, we define the operators
\begin{equation*}
	\begin{split}
		T_{\g}[v](\wt x)
		&= T_{\g}[v](x, \d, \z)
		:= \int v\left(\widetilde{x}'\right)\widetilde{\mathrm{P}}\left(d\widetilde{x}' \mid \widetilde{x},\g(\widetilde{x})\right)
		= \int v\left((x',\d',\z')\right)(\mathcal{T}_{(x,\d,\z)})_{\#}\mathrm{P}\left(dx' \mid x,\g(x,\d,\z)\right) \\
		& =\int v\left(x',\z\cdot\C(x,\g(x,\d,\z),x')+\d, \z\cdot \bb \right)\mathrm{P}(dx' \mid x,\g(x,\d,z)),
		\end{split}
\end{equation*}
and
\begin{equation*}
T[v](x,\d,\z)=\inf_{a\in D (x)}\int v(x', \z\cdot\C(x,a,x')+\d, \z\cdot\bb )\mathrm{P}(dx' \mid x,a),
\end{equation*}
whenever the integrals exist. $T$ is again called the \textit{minimal cost operator}. We may now state the main theorem of this section.

\begin{theorem}
	Let $\wt\Vw_0(x, \d, \z) := \Uw (\d)$. The following holds:
	\begin{enumerate}[label=\alph*)]
		\item For any Markovian policy $\widetilde \pi=(\g_1, \dots) \in \wPiX$, we have the cost iteration
		\begin{equation*}
		\wt \Iw_n((x,\d,\z),\wt \pi) = T_{\g_{1}}[\dots [T_{\g_{n-1}}[\wt\Vw_0]]](x,\d,\z)
		\end{equation*}
		for all $n=1,\dots,N.$
		
		\item The optimal policy is Markovian, i.e.
		$$\inf_{\widetilde\pi \in \wPiH} \Iw_{N}(\widetilde{x},\widetilde\pi) = \inf_{\widetilde\pi \in \wPiX} \Iw_{N}(\widetilde{x},\widetilde\pi).$$
		
		\item The operator $T: \Delta \to \Delta$ is well-defined, and for every $v \in \Delta$, there exists a minimizer of $T[v]$.
		
		\item We get the Bellman-style equation
		$$\wt\Vw_{n}(x,\d,\z)= T[\wt\Vw_{n-1}](x,\d,\z)
		= \inf_{a\in D(x)} \int \wt\Vw_{n-1}(x',\z\cdot \C(x,a,x')+\d, \z\cdot\bb)P(dx' \mid x,a)$$
		for all $n=1,\dots,N$.
		
		\item If $\widetilde g^*_{n}$ is a minimizer of $\wt\Vw_{n-1}$ for $n=1,\dots,N$, then $\widetilde \pi^* = (\widetilde g^*_1,\dots,\widetilde g^*_{N})$ is an optimal policy for ($\widetilde P$). In this situation, the history-dependent policy $\pi^* = (f^*_0,\dots, f^*_{N-1})$, defined by
		$$f_n^*(h_{n}):=
		\begin{dcases}
		\widetilde g^*_N(x_0, 0, 1) & \text{if } n=0, \\
		\widetilde g^*_{N-n}\left(x_n,\sum_{k=0}^{n-1} \bb^k \cdot\C(x_k,a_k,x_{k+1}), \bb^n \right) & \text{otherwise},
		\end{dcases}
		$$
		is an optimal policy for problem (P).
	\end{enumerate}
\end{theorem}

\begin{proof}
	 The proof is similar to the derivation of Theorem \ref{thm:SumsOfCostThm}. We will only prove \textit{(a)} by induction. To that end, note that
	\begin{align*}
		\wt \Iw_1((x,\d,\z), \wt \pi)
		&= \wt \E^{\wt \pi}_{x} \bigg[ \Uw(\z\cdot \C(X_0, A_0, X_1) + \d) \bigg] \\
		&= \int \Uw(\z\cdot \C(x, \wt g_0(x,\d,\z), x') + \d) P(dx' \mid x, \wt g_0(x,\d,\z)) \\
		&= T_{\wt g_1}[\wt\Vw_0](x,\d,\z).
	\end{align*}
	Let $\wt \pi^+ = (\wt g_2, \wt g_3, \dots)$. For $n = 2,\dots,N$, we get
	\begin{align*}
		&\wt \Iw_n((x,\d,\z),\widetilde{\pi})
		= \mathbb E^{\wt\pi}_{x} \left[
		\Uw \left(\z \sum_{k=0}^{n-1} \beta^k \C(X_k, A_k, X_{k+1}) + \d
		\right)
		\right] \\
		&= \int \E^{\wt\pi^+}_{x'} \left[
		\Uw \left(\z\cdot \sum_{k=0}^{n-2} \bb^{k+1}\cdot \C(X_k, A_k, X_{k+1}) + \z\cdot \C(x, \wt g_1(x,\d,\z), x') + \d
		\right)
		\right] P(dx' \mid x, \wt g_1(x,\d,\z)) \\
		&= \int \wt \Iw_{n-1}((x',\z\cdot \C(x, \wt g_1(x,\d,\z)+\d,\bb\cdot\z),\wt\pi^+) P(dx' \mid x, \wt g_1(x,\d,\z)) \\
		&= T_{\wt g_1}[\wt \Iw_{n-1}(\cdot, \wt \pi^+)](\wt x) = T_{\g_{1}}[T_{\g_{2}}[\dots [T_{\g_{n-1}}[\wt\Vw_0]]]](x,\d,\z).
	\end{align*}
	The claim follows then inductively. \qed
\end{proof}

\subsection{Infinite horizon problems}

In this section, we study the infinite horizon problem with \textit{discount factor} $\bb \in (0,1)^{\imax}$. For a vector $\boldsymbol{a}\in\mathbb{R}^{\imax},$ we will denote with $\underline{a}=\min\{a_{1},\dots,a_{\imax}\}, \overline{a}=\max\{a_{1},\dots,a_{\imax}\}.$ The notion of $T$ and $\Delta$ from the previous section is unchanged. The total cost in this situation reads as
$$\Iw_\infty(x, \pi) := \E_x^\pi \left[ \Uw \left( \sum_{n=0}^\infty \bb^n\cdot \C(X_n, A_n, X_{n+1}) \right) \right],$$
and $\wt\Vw_\infty$ is defined accordingly. We shall use the following definition:

\begin{definition}
	For a continuous function $\Fw:\mathbb{R}^{d}\rightarrow\mathbb{R},$ we define the modulus of continuity $\omega_{\Fw}(\delta,R)$ on the ball $B(0,R),$ by \begin{equation}
		\omega_{\Fw}(\delta,R)=\sup\left\{|\Fw(x)-\Fw(y)| \ \Big|\ x,y\in B(0,R), \|x-y\|_{2}<\delta\right\} 
	\end{equation}
\end{definition}

Note that the modulus of continuity is increasing in both variables, and it holds
$\lim_{\delta\rightarrow 0}\omega_{\Fw}(\delta,R)\rightarrow 0.$ 

\begin{theorem}
	Let $\underline b(\d, \z) := \Uw\left(\z\cdot \frac{\underline c}{\mathbf{1}-\bb} + \d\right)$ and $\bar b(\d, \z) := \Uw\left(\z\cdot \frac{\bar c}{\mathbf{1}-\bb} + \d\right),$ where $\frac{1}{\boldsymbol{a}}=\left(\frac{1}{a_{1}},\dots,\frac{1}{a_{\imax}}\right)$. 	
	
	\begin{enumerate}[label=\alph*)]
		\item Let $K$ be a compact subset of $\R^\imax$. Then, $T^n [\underline b] \nearrow \wt\Vw_\infty$, $T^n[ \Uw] \nearrow \wt\Vw_\infty$, and $T^n[ \bar b] \searrow \wt\Vw_\infty$ as $n \to \infty$ uniformly on $\X\times K\times (0,1)^{\imax}$.
		
		\item $\wt\Vw_\infty$ is the unique solution of
		$$\left\{
		\begin{array}{l}
			v = T[v], \\
			v \in \Delta, \\
			\underline b(\cdot, \cdot) \leq v(x,\cdot,\cdot) \leq \bar b(\cdot, \cdot) \text{ for all } x \in \X.
		\end{array}
		\right.$$
		
		\item There exists a decision rule $g^*$ that minimizes $\wt\Vw_\infty$.
		
		\item The history-dependent policy $f^* = (f_0^*, f_1^*, \dots)$ given by
		$$f_n^*(h_n) = g^*\left( x_n, \sum_{k=0}^{n-1} \bb^k c(x_k, a_k,x_{k+1}), \bb^n \right)$$
		is an optimal policy for $\Vw_\infty$.
	\end{enumerate}

\end{theorem}

\begin{proof}
\begin{enumerate}[label=\textit{{ad }\alph*)}]
\item For $n \in \N$ and $(x,\d,\z) \in \wt \X$, it holds
    \begin{equation}\label{modcont}
    \begin{split}
    &\Uw\left(\z\cdot \sum_{k=0}^\infty \bb^k\cdot \C(x_n, a_n, x_{n+1})  + \d\right)-\Uw\left(\z\cdot \sum_{k=0}^n \bb^k\cdot \C(x_n, a_n, x_{n+1})  + \d\right)\\&\leq \omega_\Uw\left(\left\|\z\cdot \bb^n \sum_{k=n}^\infty \bb^{k-n} \cdot\C(x_k, a_k, x_{k+1})\right\|_{2},\left\|\z\cdot \sum_{k=0}^\infty \bb^k\cdot \C(x_n, a_n, x_{n+1})  + \d\right\|_{2}\right)\\&\leq
    \omega_\Uw\left(\imax\overline{z} \overline{\beta}^n \frac{\bar c}{1-\overline{\beta}},\imax\overline{z}\frac{\overline{c}}{1-\overline{\beta}}  + \overline{r}\right)
    \end{split}.
    \end{equation}
Now, we get    
    \begin{equation*}
	    \begin{aligned}
	        \wt\Vw_n(x,\d,\z)
	        &\leq \wt\Iw_{n,\wt\pi}(x,\d,\z)  \leq \wt\Iw_{\infty,\wt\pi}(x,\d,\z) = \E_x^{\wt\pi} \left[ \Uw\left(\z\cdot \sum_{k=0}^\infty \bb^k\cdot \C(X_n, A_n, X_{n+1})  + \d\right) \right] \\
	        &\leq \E_x^{\wt\pi} \left[ \Uw\left(\z\cdot \sum_{k=0}^n \bb^k\cdot \C(X_n, A_n, X_{n+1}) + \d\right) \right]  +  \omega_\Uw\left(\imax\overline{z} \overline{\beta}^n \frac{\bar c}{1-\overline{\beta}},\imax\overline{z}\frac{\overline{c}}{1-\overline{\beta}}  + \overline{r}\right)\\
	        &\leq \wt\Iw_{n,\wt\pi}(x,\d,\z) + \underbrace{\omega_\Uw\left(\imax\overline{z} \overline{\beta}^n \frac{\bar c}{1-\overline{\beta}},\imax\left(\overline{z}\frac{\overline{c}}{1-\overline{\beta}}  + \overline{r}\right)\right)}_{=: \varepsilon_n(x,\d,\z)}.
	    \end{aligned}
	\end{equation*}
Since $K$ is compact, there exists $R > 0$ such that $\bar d < R$. Then, we have 
$$\varepsilon_n(x,\d,\z)\leq \omega_\Uw\left(\imax\overline{\beta}^n \frac{\bar c}{1-\overline{\beta}},\imax\frac{\overline{c}}{1-\overline{\beta}}  + R\right),$$
and since $\overline{\beta} \in (0,1)$, we have $\varepsilon_n \searrow 0$ uniformly. 
% on sets of the form $X\times K\times (0,1)^{\imax},$ where $K$ is a compact subset of $(\mathbb{R}^{+})^{\imax}.$
Because $\wt\pi$ was arbitrary, the previous inequality also holds for the infimum, i.e.
\begin{equation} \label{eqn:withinfimum}
	\wt\Vw_n(x,\d,\z) \leq \wt\Vw_\infty(x,\d,\z) \leq \wt\Vw_n(x,\d,\z) + \varepsilon_n(x,\d,\z),
\end{equation}
    and therefore $\wt\Vw_n \nearrow \wt\Vw_\infty$.

    Recall that $\C$ is componentwise bounded by $\underline c, \bar c > 0$, and therefore, independent of the process $(X_n), (A_n)$, the infinite time cost $\sum_{n=0}^\infty \bb^n\cdot \C(X_n, A_n, X_{n+1})$ is componentwise  bounded by $\frac{\underline c}{\mathbf{1}-\bb}, \frac{\bar c}{\mathbf{1}-\bb}$. We have therefore $\underline b \leq \wt\Vw_\infty \leq \bar b$. Since $T$ is increasing, we have with the previous result $\wt\Vw_{n+1} = T[\wt\Vw_n] \leq T[\wt\Vw_\infty]$, i.e. $\wt\Vw_\infty \leq T[\wt\Vw_\infty]$. Since $\z\in  (0,\infty)^{\imax}$, we observe that for every triple $(x,\d,\z)$ and $a\in\D(x),$ we have 
    \begin{equation}
    \begin{split}
\varepsilon'_{n}(x,\d,\z, a)&:=\int\varepsilon_n(x',\z\cdot\C(x,a,x')+\d,\z\cdot\bb)\mathrm{P}(dx' \mid x,a)\\&\leq \omega\left(\imax\overline{z} \overline{\beta}^{n+1}\frac{\bar c}{1-\overline{\beta}},\imax\left(\overline{z}\overline{\beta}\frac{\overline{c}}{1-\overline{\beta}}  + \overline{r}+\overline{z}\, \overline{c}\right)\right)\\&\leq
\omega\left(\imax\overline{z} \overline{\beta}^{n+1}\frac{\bar c}{1-\overline{\beta}},\imax\overline{z}\frac{\overline{c}}{1-\overline{\beta}} \right)=\varepsilon_{n+1}(x,\d,\z).
    \end{split}
    \end{equation}   
  Now, we get with (\ref{eqn:withinfimum})
  \begin{equation}
  \begin{split}
  T[\wt\Vw_\infty](x,\d,\z)&\leq  T[\wt\Vw_n + \varepsilon_n](x,\d,\z) \leq  T[\wt\Vw_n](x,\d,\z) +\sup_{a\in\D(x)}\varepsilon'_{n}(x,\d,\z, a)\\&\leq \wt\Vw_{n+1}(x,\d,\z)+\varepsilon_{n+1}(x,\d,\z),
  \end{split}
  \end{equation}
 but the last term converges to zero as $n \in \infty$. Therefore we have, $\wt\Vw_\infty \geq T[\wt\Vw_\infty]$, i.e. $\wt\Vw_\infty = T[\wt\Vw_\infty]$.
    
    We next show that $T^n[ \bar b] \searrow \wt\Vw_\infty$, and $T^n [\underline b] \nearrow \wt\Vw_\infty$ as $n \to \infty$. First, observe that
    \begin{align*}
    	T[ \bar b] (x, \d, \z)
    	&= \inf_{a \in D(x)} \int \Uw\left( \z\cdot \bb \cdot\frac{\bar c}{1-\bb} + \z \cdot\C(x,a,x') + \d \right) P(dx'|x,a) \\
    	&\leq \Uw \left( \z\cdot \frac{\bar c}{\mathbf{1}-\bb} + \d \right) \\
    	&\leq \bar b (\d,z),
    \end{align*}
    and the same holds true for
    $$T [\underline b](x,\d,z) \geq \underline b(\d,z).$$
    Since $T$ is increasing, the sequences $(T^n [\bar b])_n$ and $(T^n [\underline b])_n$ are pointwise monotone and bounded, and therefore their pointwise limit exists. By iteration,
    \begin{align*}
    	T^n[\Uw](x,\d,\z) &= \inf_{\pi \in \wPiX} \E_x^\pi \left[ \Uw \left( \z \cdot\sum_{k=0}^{n-1} \bb^k\cdot \C(X_k, A_k, X_{k+1}) + \d \right) \right] \\
    	T^n\left[\bar b\right](x,\d,\z) &= \inf_{\pi \in \wPiX} \E_x^\pi \left[ \Uw \left( \z\cdot \bb^n \frac{\bar c}{\mathbf{1}-\bb} + \z \cdot\sum_{k=0}^{n-1} \bb^k\cdot \C(X_k, A_k, X_{k+1}) + \d \right) \right].
    \end{align*}
   We obtain
    \begin{align*}
    	0
    	&\leq T^n[\bar b](x,\d,\z) - T^n[\underline b](x,\d,\z)\hspace{20pt} \text{ by monotonicity of} \hspace{2pt}T \hspace{2pt}\text{and} \hspace{2pt}T(0)=0\\
    	&\leq T^n[\bar b](x,\d,\z) - T^n[\Uw](x,\d,\z)  \hspace{20pt}\text{by monotonicity of} \hspace{2pt}T \hspace{2pt}\text{and} \hspace{2pt} \mathcal{U}(\d)\leq\underline{b}(x,\d,\z)\\
    	&\leq \sup_{\pi \in \Pi_{\wt \H}} \E_x^\pi \left[ \Uw\left( \z\cdot \bb^n \frac{\bar c}{\mathbf{1}-\bb} + \z \sum_{k=0}^{n-1} \bb^k\cdot\C(X_k, A_k, X_{k+1}) + \d \right)  -  \Uw\left( \z\cdot \sum_{k=0}^{n-1} \bb^k \cdot \C(X_k, A_k, X_{k+1}) + \d \right) \right]\\&= \varepsilon_n(x,\d,\z).
    \end{align*}
    For $n \to \infty$, we obtain
    $$\lim_{n \to \infty} T^n \underline [b] = \lim_{n \to \infty} T^n [\bar b]= \lim_{n \to \infty} T^n [\Uw] = \wt\Vw_\infty,$$
    uniformly on compact sets. This proves (a).
    
    \item $\wt\Vw_\infty$ is lower semi-continuous as a uniform limit on sets of the form $\X\times K\times (0,1)^{\imax},$ where $K$ is a compact subset of $\R^\imax,$  of lower semi-continuous function. As proved in Theorem \ref{thm:SumsOfCostThm}, $T^n [\bar b](x, \cdot, \cdot)$ is continuous and componentwise monotonous for all $x \in \X$. Since $T^n [\bar b] \searrow \wt\Vw_\infty$, $\wt\Vw_\infty(x,\cdot,\cdot)$ is upper semi-continuous and therefore continuous, and also preserves the same monotonicities for each $x \in \X$. We have thereby shown $\wt\Vw_\infty \in \Delta$.
    
	It remains to show the uniqueness. To that end, suppose that there is $v \in \Delta$, $v \neq \wt\Vw_\infty$ such that $v = T[v]$ with $\underline b \leq v \leq \bar b$. Then, because $T$ is increasing, $T^n [\underline b] \leq T^n [v] = v \leq T^n [\bar b]$, and with $n \in \infty$, we get $\wt\Vw_\infty \leq v \leq \wt\Vw_\infty$, i.e. $v = \wt\Vw_\infty$, a contradiction. This proves (b).
	
	\item The claim follows similar to Theorem \ref{thm:SumsOfCostThm}.
	
	\item Since $\wt\Vw_\infty(x,y,z) \geq \Uw(y)$, we obtain
	$$\wt\Vw_\infty = \lim_{n \to \infty} T^n_{g^*} [\wt\Vw_\infty] \geq \lim_{n \to \infty} T^n_{g^*} [\Uw] = \lim_{n \to \infty} \wt\Iw_n(\cdot,(g^*, g^*, \dots)) = \wt\Iw_\infty(\cdot, (g^*, g^*, \dots)) \geq \wt\Vw_\infty,$$ Hence $f^*$ is optimal for $\wt\Iw_\infty$. This finally proves (d). \qed
\end{enumerate}
\end{proof}

\section{Numerical example}
\subsection{Task design}
To illustrate our method, we present a generalized version of the repetitive Tiger problem \citep{kaelbling1998planning}. In the classic Tiger problem, a decision-maker is faced with two close doors, behind one of which is a tiger (punishment) and behind the other is a treasure (reward). In the original version, the agent can either open a door or listen to tiger sound in order to gain more information about the true place of the tiger. The sound signal however, is not fully reliable and with a smaller probability (20\%) it can be heard from the wrong door. Each time that agent opens a door, it takes the reward/punishment and the problem resets. By resetting the problem, the position of the tiger and the treasure would be determined randomly and would remain fixed for the entire next trial.\medskip

The version we will discuss here, has been generalized in three aspects: First, the constraint of deterministic state space has been relaxed and the position of the tiger can change during a trial. Here, there is a 10 \% chance to change its position at the start of each epoch.
In the second extension, the immediate reward (punishment) of opening the doors would not be observable for the agent during the repeats of the problem. Therefore, at any time before the end of the all trials, the agent can only have an estimation about its gains. At the end of all trials, the agent would observe the whole accumulated reward and punishments.  And third, the action space is expanded in a way that makes the agent able to not only choose the correct option (treasure door) but also bet on its own decision in different levels of investment. Here, we define that the agent can open each door either conservatively (low stake actions), to gain a lower amount of reward and punishment (low stake rewards: $Reward_{corr\_low}$ and $Reward_{incorr\_low}$), or rush to the doors (high stake actions), to gain a bigger reward if it is the treasure, and to take more damage if the tiger is behind the door. (high stake rewards: $Reward_{corr\_High}$ and $Reward_{incorr\_High}$). \medskip

The underlying MDP of the experiment is depicted in fig \ref{fig:fig1}. Here, the probabilities of getting observations ($O_{tiger\_right}$ and $O_{tiger\_left}$) are depend on the agent's actions and the successive new state. As mentioned before, by doing the \textit{listen} action there is a smaller probability that agent takes a false signal. Also, by doing each of \textit{open} actions (regardless of their correctness or their stake type) the MDP will be reset and a random signal, with probability of 0.5 for pointing to each state, would be received. In other words, the signal which received after the re-initializing the problem, is uninformative in purpose of detecting the tiger's position. The observation function of the experiment's POMDP is shown in table 1. \medskip

\begin{figure}
    \centering
    \includegraphics[width=0.9\textwidth]{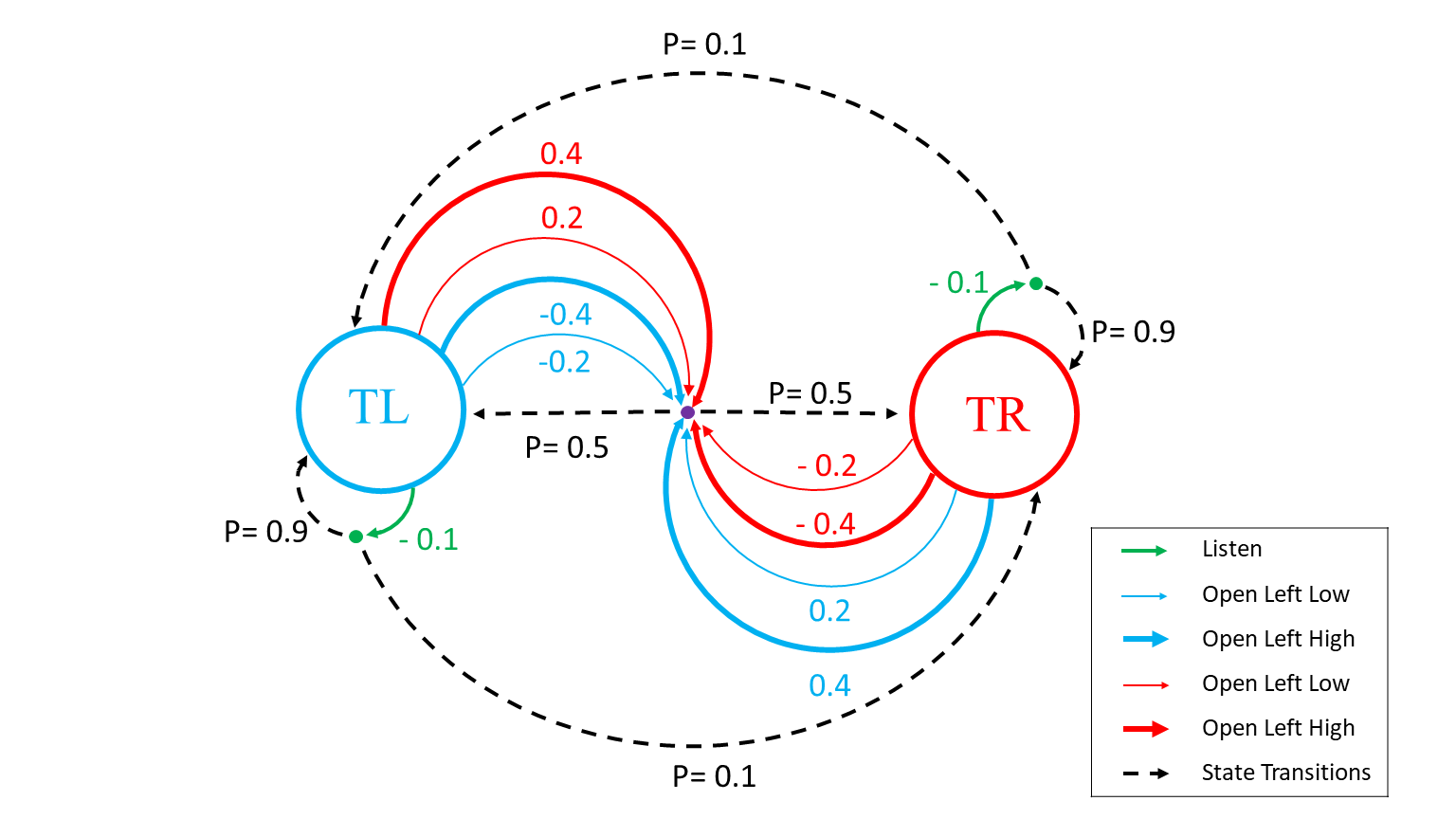}
    \caption{\textbf{ MDP of the Extended Tiger task.} The states of the MDP are Tiger\_right(TR) and Tiger\_left(TL). After each \textit{opening} action (red and blue arrows), the environment would set to a state randomly and based on either choosing tiger door or treasure door as well as choosing either high stake or low stake action (thickness of the arrows), the non-observable reward will receive by agent. By doing \textit{listen} action, the agent will pay a small cost and takes an informative signal about the position of the tiger. There is also a small chance that tiger changes its position. }
    \label{fig:fig1}
\end{figure}

\begin{figure}
    \centering
    \includegraphics[width=0.9\textwidth]{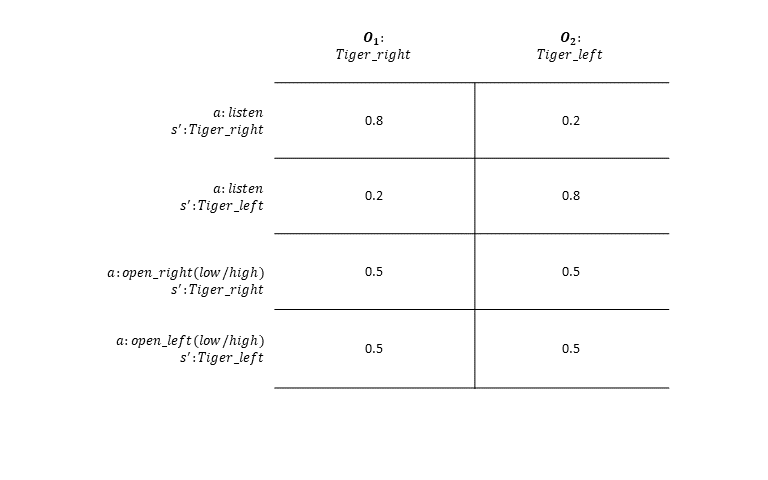}
    \caption{\textbf{Tabular representation of the Observation Function.}}
    \label{fig:table1}
\end{figure}
\subsection{Simulation results}
In order to illustrate the competency of our method in replicating different risk sensitive behaviors, in this section we have tested our method on the extended tiger task by using four different utility functions which are composed by linear combinations of of exponential functions. Three out of them are adjusted to address risk-neutral, risk-seeking and risk-aversion decision-making. It means the utility functions are (near-)linear, convex and concave in the interval of possible rewards respectively. It should be mentioned that as in our method $e^0$ is not a valid term, linear combinations cannot replicate an exact linear function ($\frac{d^2U}{dx^2} \ne 0$). However, for the sake of being more intuitive we used a utility function with an infinitesimal second derivative in the interval of rewards to show the ability of the model to mimic different patterns of risk sensitivity (fig2.a). Last but not least, we have tested our model on the task with a sigmoid utility function. As mentioned before, it is assumed that human uses S-shaped utility functions, like sigmoid function, in face with losses and gains. Therefore, regarding computational modeling of behavior, it is a crucial ability for a risk sensitive model to mimic such functions. Like the linear case, sigmoid function cannot be expressed by linear combinations of exponential functions. However, we can approximate the sigmoid in a specific interval by using a combination which contains enough numbers of exponential functions. Here, We fitted weighted sum of five exponential terms: $0.381*e^{0.2906}$, $0.404*e^{0.2876}$, $-0.427*e^{-0.0091}$,  $-0.182*e^{0.6537}$, $0.322*e^{0.2982}$. \\

\begin{figure}
    \centering
    \includegraphics[width=0.99\textwidth]{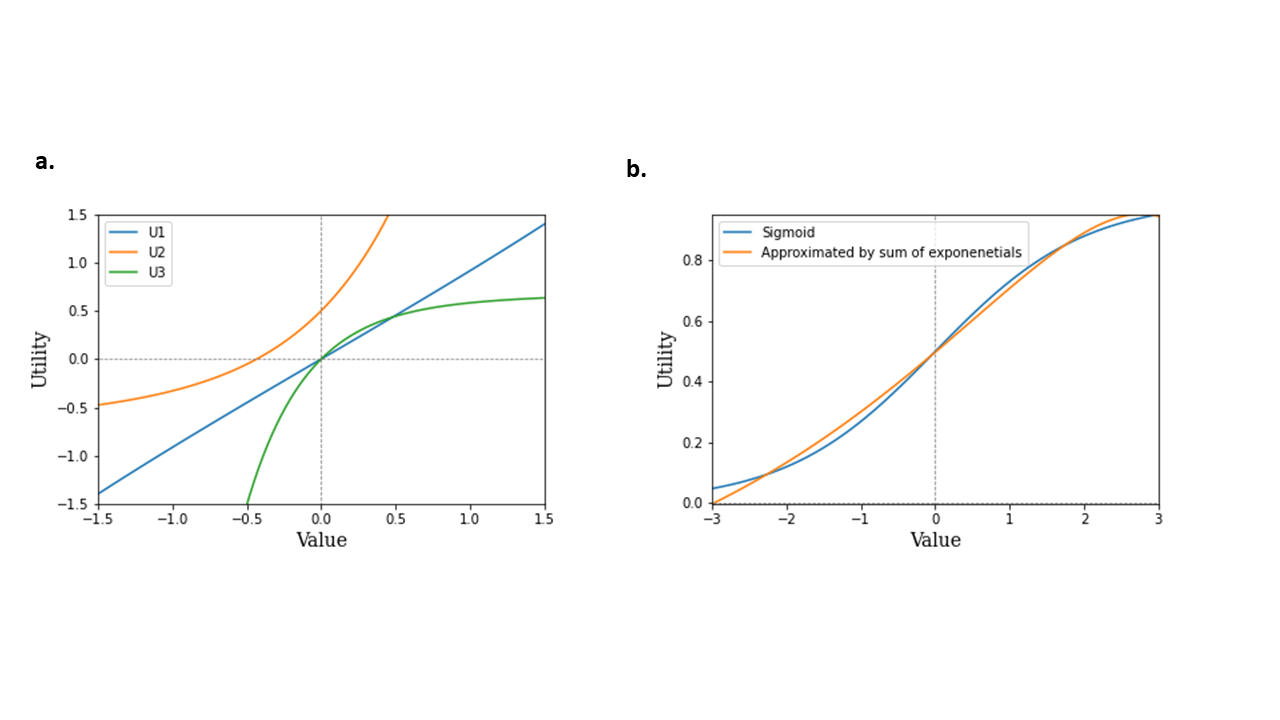}
    \caption{\textbf{Utility functions. a.} Utility functions produce risk\_neutral, risk\_seeking and risk\_aversion decisions in the showed interval respectively. risk\_neutral: $U_1= 1.5e^{0.3}-1.5e^{-0.3}$, risk\_seeker: $U_2= e^{1.5}-0.5e^{-0.1}$, risk\_avers: $U_3= 0.6e^{0.05}-0.6e^{-2.5}$. \textbf{b. }Approximation of Sigmoid function by weighted sum of five exponential functions. }
    \label{fig:utilities}
\end{figure}

The simulation results have been presented in table 2. The results clearly show the effect of utility functions' shape on the risk attitude of the simulated agent. In table 2 we only present selected actions in trials with depth of planning equal to either one or two steps. As we have used deterministic greedy policy in our simulations, choosing between different action types (listening, low stake door openings and high stake door openings) in plannings with maximum depth of one or two only depends on environment dynamics, utility function, discount factor and initial wealth. In other words, in planning with depth one or two, choosing the type of actions(and not their directions) is independent from the observations from the environment side. Therefore, we can easily fix the other dynamics and examine only the effect of utility functions. In this simulations, we have fixed the environments dynamics to the above-mentioned values, with no discounting and the initial wealth equals to zero. \medskip
\\
In the extended tiger problem, one can assume that high/low stake opening actions represent riskiness of the decisions. While the expected reward of them are equal, the deviation of outcomes in high stake cases are higher. \ref{fig:table2} shows that in the maximum planning-depth of two, risk-neutral agent prefers to gather more information (and pay its cost) in the first step, and then in the second step open the door with higher probability of being treasure in high stake mode. However, the risk averse agent($u_3$) prefers to do the second action more conservatively and open the door in the low stake mode while in contrast with them, the risk-seeker agent ($u_2$) prefers to perform risky actions in each step. The sigmoid-agent also behaves like the risk-neutral case, however it should be considered that S-shaped utilities make agents risk-averse toward positive accumulated outcomes and risk-averse in face with negative valuations of total expected rewards. In one step planning conditions, paying the certain cost of listening rather than doing a risky action with higher expected return seems irrational for all of used utility functions. However, risk-averse and risk-neutral cases prefer to choose low stake actions in a fifty-fifty situation while the risk-seeker and sigmoid agents prefer to risk more.

\begin{figure}
    \centering
    \includegraphics[width=0.99\textwidth]{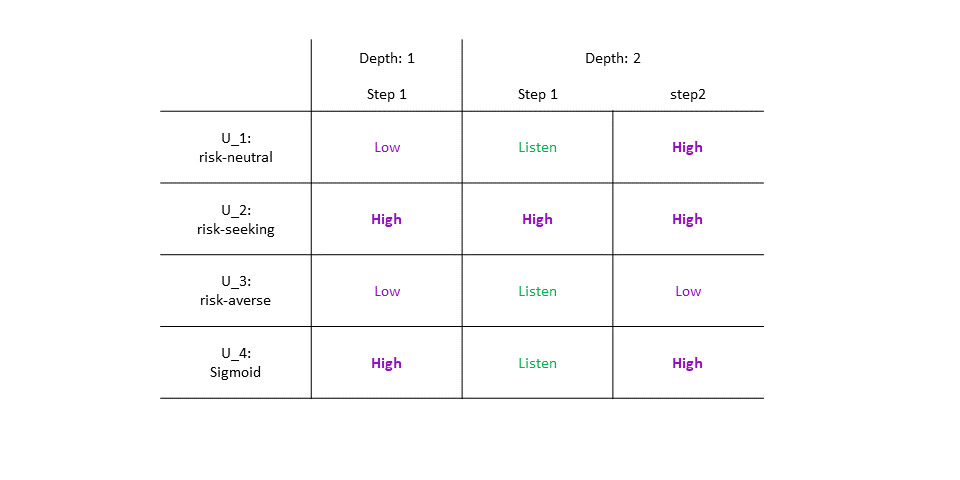}
    \caption{\textbf{Actions with best value for each step under different utility functions.}}
    \label{fig:table2}
\end{figure}

\section {Discussion}

The method we introduced works for problems which have finite set of states while using any increasing utility function. Our method calculates the exact utility values in case of weighted sum of exponential utility functions and approximate them for any other monotone function, contrasting \citet{bauerle2017a} which is more general and doesn't need to approximate values. However the resulting augmented state space in Mutlivariate utility method is $\mathcal{P}(\mathcal{X})^{i_{\max}}\times\mathcal{Y}\times\mathbb{R}^{i_{\max}}\subset \mathbb{R}^{|\mathcal{X}|\times 2 \times i_{\max}}\times\mathcal{Y}$, where $i_{\max}$ is the number of exponential functions that make up the utility function, see \eqref{sumexp}. In \citet{bauerle2017a}, the resulting state space is  $\mathcal{P}(\mathcal{X}\times\mathbb{R})$ which is an infinite dimensional space, and even in cases where the wealth space is discretized appropriately, one ends up with a dimension of  $|\mathcal{X}|\cdot$(partition size). Our method therefore has a clear computational advantage when $i_{\max}$ is small. In the general case of approximating utility function, the lower dimensionality of Multivariate method brings the trade-off between accuracy of approximation and computational tractability to attention. One can expect that by increasing the number of exponential terms in the approximated utility function, the accuracy of approximated utility values would improve (become more similar to their exact non-approximated values) but in the cost of an increase in state space dimensionality. Moreover by increasing the depth of planning, the method introduced by \citet{bauerle2017a} would also face with the trade-off between lack of accuracy and increase of state space complexity in case of using partitioned wealth-axis. Because, when the maximum depth of planning grows the possible amounts of wealth would also increase. Both mentioned accuracy/complexity trade-offs are heavily dependent on the dynamics of the problem as well as the utility function and can be subject of further studies.
Last but not least, our proposed model is eligible to apply on problems which would be defined in a multi-variate manner. In this work we only discussed the ability of the Multivariate model to address monotone utility functions in a class of problems which only have one objective (wealth), however the problems with different separate running costs like: resource allocation in different governmental sectors or maximizing the overall utility of an economic actor while she uses different diminishing marginal utility functions in different goods or aspects of life are another area that our method can address and could be investigated more in terms of computational efficiency. 
\\

\bibliographystyle{apacite}
\bibliography{main}
%\printbibliography
\end{document}